\documentclass[11pt,reqno]{amsart}
\usepackage[width=150mm,height=220mm,
centering,
]{geometry}
\usepackage{mathrsfs}
\usepackage{amsxtra}
\usepackage{amsmath,amsfonts,amssymb}
\usepackage{upgreek}
\usepackage{xcolor}
\usepackage[numbers,sort&compress]{natbib}
\usepackage{csquotes}
\numberwithin{equation}{section}
\allowdisplaybreaks
\ifcsname proof\endcsname
  \else

\fi

\newtheorem{theorem}{Theorem}[section]
\newtheorem{proposition}[theorem]{Proposition}
\newtheorem{lemma}[theorem]{Lemma}

\newtheorem{definition}[theorem]{Definition}
\newtheorem{remark}[theorem]{Remark}
\newtheorem{corollary}[theorem]{Corollary}
\newtheorem{conjecture}{Conjecture}[section]

\begin{document}

\parindent 2em
\title[Pointwise Convergence of Schr\"{o}dinger Operators]
{Pointwise Convergence of Schr\"{o}dinger Operators in Bessel Potential Spaces}
\author[Y. Pan]{Yucheng Pan}
\address{
School of Mathematical Sciences and LPMC\\
Nankai University\\
Tianjin\\
China}

\email{panyucheng@mail.nankai.edu.cn}

\author[W. Sun]{Wenchang Sun}
\address{
School of Mathematical Sciences and LPMC\\
Nankai University\\
Tianjin\\
China}

\email{sunwch@nankai.edu.cn}

\author[J. Tan]{Jiheng Tan}
\address{
School of Mathematical Sciences and LPMC\\
Nankai University\\
Tianjin\\
China}

\email{tanjh@mail.nankai.edu.cn}

\thanks{Corresponding author: Wenchang Sun.}


\begin{abstract}
We study the pointwise convergence of solutions to the free Schr\"{o}dinger equation with initial data in the Bessel potential spaces $L_s^p(\mathbb{R}^n)$.
We establish new sufficient regularity indices for pointwise convergence across the full range $1 \leq p < \infty$, and demonstrate via counterexamples that these indices are sharp for all $1 \leq p \leq 2$ in one dimension, as well as for $p=1$ 
or $p$  large enough in higher dimensions.
The proofs rely on the high-dimensional stationary phase method.
\end{abstract}

\thanks{This work was supported by the
National Natural Science Foundation of China (No. 12571104 and 12271267).}

\maketitle
\date{\today}

\section{Introduction}\label{s1}
The free Schr\"{o}dinger equation plays a central role in the theory of dispersive partial differential equations.
In the Euclidean space $\mathbb{R}^{n}$, the initial value problem is given by
\begin{equation}\label{seq}
\begin{cases}
i\partial_t u(x,t) + \Delta u(x,t) = 0, \qquad (x,t) \in \mathbb{R}^{n} \times \mathbb{R}^{+}, \\
u(x,0) = f(x).
\end{cases}
\end{equation}
The formal solution to \eqref{seq} can be expressed as
\begin{equation*}
e^{it\Delta}f(x) = \frac{1}{(2\pi)^{n}}\int_{\mathbb{R}^n} e^{ix \cdot \xi - it|\xi|^2} \widehat{f}(\xi) \mathrm{d}\xi,
\end{equation*}
where $\widehat{f}$ is the Fourier transform of $f$.

A fundamental problem, posed by Carleson \cite{C} in 1979, is to determine the optimal regularity index $s>0$ such that
\begin{equation*}
\lim_{t\to 0^{+}} e^{it\Delta}f(x)=f(x) \quad \text{a.e.}
\end{equation*}
whenever $f\in H^{s}(\mathbb{R}^{n})$. By the standard density argument, this pointwise convergence problem reduces to establishing the local boundedness of the associated maximal operator:
\begin{equation*}
\left\|\sup_{0<t<1}\left|e^{it\Delta}f\right|\right\|_{L^p(B(0,1))}\lesssim \|f\|_{H^s(\mathbb{R}^n)}.
\end{equation*}

Historically, the study of the maximal Schr\"{o}dinger operator on $H^{s}(\mathbb{R}^{n})$ has focused on determining the optimal regularity index. In one dimension, Carleson \cite{C} proved that $s\geq 1/4$ is sufficient for pointwise convergence, which was later shown to be sharp by Dahlberg and Kenig \cite{DK}.

In higher dimensions, Sj\"{o}lin \cite{S87} and Vega \cite{V} independently established the sufficiency for $s>1/2$. By exploiting the Fourier restriction theory, Bourgain \cite{B95} provided an alternative proof for $s>1/2$ when $n=2$. This index was subsequently improved to $s>\kappa$, $20/41<\kappa<41/84$ by Moyua, Vargas and Vega \cite{MVV}.
By introducing the method of bilinear restriction estimates, Tao and Vargas \cite{TV} further lowered the requirement to $s>15/32$ for $n=2$. Building upon their framework, Lee \cite{L} improved the two dimensional index to $s> 3/8$.
In 2013, Bourgain \cite{B13} established pointwise convergence for $1/2-1/(4n)$ and demonstrated its failure for $s<1/2-1/n$ when $n\geq 4$.
Later, Bourgain \cite{B16} constructed geometric counterexamples showing that pointwise convergence fails for $s< n/(2n+2)$ in all dimensions.
Utilizing $\ell^2$-decoupling and polynomial partitioning, the optimal sufficiency condition $s>n/(2n+2)$ was ultimately established by Du, Guth, and Li \cite{DGL} for $n=2$, and by Du and Zhang \cite{DZ} for all $n\geq 3$.

While the pointwise convergence problem is resolved
for initial data in $H^s(\mathbb{R}^n)$ up to the
endpoint $s=n/(2n+2)$, the analogous question in the general Bessel potential spaces $L_s^p(\mathbb{R}^n)$ ($p \neq 2$) remains. These function spaces are defined as follows.

\begin{definition}\cite[Definition 1.3.2]{G250}
Let $s \in \mathbb{R}$ and $1\leq p\leq \infty$. The Bessel potential space $L_s^{p}(\mathbb{R}^{n})$ is defined as the space of all tempered distributions $f \in \mathcal{S}'(\mathbb{R}^{n})$ such that $((1+|\xi|^{2})^{s/2}\widehat{f})^{\vee} \in L^p(\mathbb{R}^{n})$. For such distributions $f$, we define the norm
\begin{equation*}
\|f\|_{L_s^{p}(\mathbb{R}^{n})}=\|((1+|\cdot|^{2})^{s/2}\widehat{f})^{\vee}\|_{L^p(\mathbb{R}^{n})}.
\end{equation*}
\end{definition}

By definition, it is clear that $L_s^2(\mathbb{R}^n)$ coincides with the classical Sobolev space $H^s(\mathbb{R}^n)$ for any $s \in \mathbb{R}$. We note that if $f_s=((1+|\cdot|^{2})^{s/2}\widehat{f})^{\vee}$, then
\begin{equation*}
f=(\widehat{f_s}\widehat{G_s})^{\vee}=f_s*G_s:=\mathcal{J}^{-s}f_s,
\end{equation*}
where $G_s$ is the Bessel kernel given by
\begin{equation*}
G_s=((1+|\cdot|^{2})^{-{s}/{2}})^{\vee}.
\end{equation*}

For the case $p>2$, the standard approach to the pointwise convergence problem shifts from bounding the maximal operator directly to establishing local smoothing estimates.
\begin{conjecture}[Local Smoothing Conjecture]\label{local_smoothing}
Fix a number $\alpha\in (1,\infty)$. For $p>2(n+1)/n$, it holds that
\begin{equation}\label{local}
\left\| e^{it\Delta^{\alpha/2}} f\right\|_{L^{p}(\mathbb{R}^{n}\times [0,1])}\lesssim \|f\|_{L_s^{p}(\mathbb{R}^{n})}.
\end{equation}
for every $s>\alpha s_{0}=\alpha(n(1/2-1/p)-1/p)$.
\end{conjecture}
Following the arguments in \cite[Sections 4 and 5]{RS}, the local smoothing estimate \eqref{local} implies the corresponding maximal bound:
\begin{equation*}
\left\|\sup_{0<t<1}\left|e^{it\Delta^{\alpha/2}} f\right|\right\|_{L^{p}(B(0,1))}\lesssim \|f\|_{L_{s+\alpha/p}^{p}(\mathbb{R}^{n})}.
\end{equation*}
Although this deduction incurs an $\alpha/p$ loss in regularity, the local smoothing perspective remains an effective mechanism for obtaining  pointwise convergence results when $p > 2$.

For the standard Schr\"{o}dinger equation ($\alpha=2$), Rogers \cite{R} initially established \eqref{local} for $p>2+4/(n+1)$ and $s>2s_0$. This was later extended by Rogers and Seeger \cite{RS} to all fractional orders $\alpha>1$ with $s\geq \alpha s_0$.
For $\alpha=2$, Lee, Rogers, and Seeger \cite{LRS} further improved the range of integrability to $p>2+4/n-2/(n^2-n(4-q_0)/(q_0-2))$ for $q_0\in (2,2(n+3)/(n+1))$, and they also established the necessary conditions $s\geq 2s_0$ and $s> 0$ via scaling arguments.

With the development of modern harmonic analysis tools, Guo, Roos, and Yung \cite{GRY} obtained local smoothing estimates for $s>\alpha s_0$ in the range
\begin{equation*}
p>\begin{cases}
\frac{2n+3}{n} &\text{if}\,\, n\equiv 0 \pmod 3,\\
\frac{4(n+2)}{2n+1} &\text{if}\,\, n\equiv 1 \pmod 3,\\
\frac{2(4n+7)}{4n+1} &\text{if}\,\, n\equiv 2 \pmod 3.
\end{cases}
\end{equation*}
Gao, Miao, and Zheng \cite{GMZ} subsequently improved this result to $p>2(3n+4)/(3n)$ for even $n$, and $p>2(3n+5)/(3n+1)$ for odd $n$. Later, Gan, Oh, and Wu \cite{GOW} proved the following result.
\begin{proposition}
For $\alpha$ in the range below
\begin{equation*}
\begin{cases}
\alpha\in (0,1) \cap (1,\infty), \qquad \text{for} \,\, n=3,\\
\alpha \in (1,\infty), \qquad \text{for} \,\, n>3
\end{cases}
\end{equation*}
and
\begin{equation}\label{p*}
p\geq 2+p_n(k)=2+\frac{6}{2n+(k-1)\prod_{i=k}^{n}2i/(2i+1)}, \quad 2\leq k\leq n,
\end{equation}
\eqref{local} holds for all $s> \alpha s_0$.
\end{proposition}
By adapting the method of Gan, Oh, and Wu \cite{GOW}, Gao, Li, and Wang \cite{GLW} showed that the H\"{o}rmander conjecture implies the local smoothing conjecture.
Furthermore, Guo, Wang, and Zhang \cite{GWZ} obtained results related to the H\"{o}rmander conjecture and showed that their estimates imply the results of \cite{GOW}.

Parallel to the continuous-time problem, the convergence of Schr\"{o}dinger means along discrete sequences $t_n \to 0$ has also been studied.
For the Bessel potential spaces $L_s^p$ with $1 \leq p \leq 2$, this sequential convergence was investigated by Sj\"{o}lin and Str\"{o}mberg \cite{SS}, who established sufficient conditions based on the summability properties of the time sequence.

However, the continuous-time maximal estimate for $1 \leq p \leq 2$ does not follow from these discrete results. Sequential convergence relies essentially on the specific summability properties of the chosen time sequences, thereby bypassing the continuous oscillatory dynamics captured by the supremum over an interval.
Furthermore, the local smoothing framework utilized for $p > 2$ fails for the range $1 \leq p \leq 2$ due to known geometric obstructions associated with Fourier restriction \cite[Page 4]{D}.

In this paper, we investigate the continuous-time maximal operator on $L_s^p$ spaces for the full range $1 \leq p < \infty$. By applying the stationary phase method to analyze the associated oscillatory kernels in higher dimensions, we establish new regularity conditions for pointwise convergence. Our main results are stated as follows.

\begin{theorem}\label{t1}
For a given $1\leq p\leq 2$, define the regularity index $s(p)$ as follows,
\begin{equation*}
s(p)= n\left(\frac{1}{p}-\frac{1}{2}\right)+\frac{n}{n+1}
\left(1-\frac{1}{p}\right).
\end{equation*}
Suppose that $f\in L_s^{p}(\mathbb{R}^{n})$. 
Then the maximal estimate
\begin{equation}\label{conv1}
\left\|\sup_{0<t<1}\left|e^{it\Delta}f\right|\right\|_{L^p(B(0,1))}\lesssim \|f\|_{L_s^p(\mathbb{R}^{n})}
\end{equation}
holds in the following cases:
\begin{enumerate}
\item For $n \geq 2$, the estimate \eqref{conv1} holds when $s> s(p)$ for $1 < p \leq 2$, and when $s \geq n/2$ for $p=1$.
\item For $n = 1$, the estimate \eqref{conv1} holds when $s \geq {1}/{2p}$ for all $1 \leq p \leq 2$.
\end{enumerate}
\end{theorem}

For $p\geq 2$, applying the complex interpolation between the result of Gan, Li and Wu \cite{GLW} and the sharp results \cite{C,DGL,DZ} for $p=2$, we deduce the following proposition.
\begin{theorem}\label{t2}
For a given $p\geq 2$ and $f\in L_s^{p}(\mathbb{R}^{n})$, the maximal estimate \eqref{conv1} holds for all $s>s(p)$, where
\begin{equation*}
s(p)=\begin{cases}
n(1-\frac{2}{p}), &p>2+p_n,\\
n(1-\frac{2}{p})+\frac{n}{n+1}\frac{2+p_n-p}{pp_n}, &2\leq p\leq 2+p_n,
\end{cases}
\end{equation*}
and $p_n$ satisfies \eqref{p*}.
\end{theorem}

As $p \to \infty$, the regularity index $s(p)$ in Theorem \ref{t2} approaches $n$.
This necessary condition $s > n$ at the $L^\infty_s(\mathbb{R}^{n})$ endpoint is sharp, which we demonstrate via the following divergence result.
\begin{theorem}\label{infty_s}
For any fixed $0 < t_0 < 1$ and $s \leq n$, there exists a tempered distribution $f \in L_s^\infty(\mathbb{R}^n)$ such that the natural integral representation of $e^{it_0\Delta}f(0)$ diverges absolutely.
\end{theorem}

Furthermore, inspired by Rogers and Seeger \cite{RS}, we can construct a counterexample to show that the regularity index $s>n(1-2/p)$ is necessary for $2\leq p<\infty$. This implies that the pointwise convergence condition in Theorem \ref{t2} is sharp for $p> 2+p_n$.

\begin{theorem}\label{p_s}
Suppose that $2\leq p<\infty$.
Then there exists some $f\in L_s^{p}(\mathbb{R}^{n})$, such that the weak-type maximal estimate
\begin{equation}\label{p>2}
\left\|\sup_{0<t<1}\left|e^{it\Delta} f\right|\right\|_{L^{p,\infty}(B(0,1))}\lesssim \|f\|_{L_{s}^{p}(\mathbb{R}^{n})}
\end{equation}
is not true for $s< n(1-2/p)$.
\end{theorem}

For the pointwise convergence of Schr\"odinger operators when $1\leq p\leq 2$,
inspired by the counterexamples given by Dahlberg and Kenig \cite{DK} and Vega \cite{V},
we obtain the following negative result, which shows that our pointwise convergence results are sharp for $1\leq p\leq 2$ in one dimension, and at the endpoints $p=1$ in higher dimensions.
\begin{theorem}\label{p<2}
  Let $1\leq p<\infty$, $1\leq q<\infty$ and $n\geq 1$. Then the inequality
\begin{equation}\label{eq:s2t1}
  \left\|\sup_{0<t<1}\left|e^{it\Delta} f\right|\right\|_{L^{q,\infty}(B(0,1))}
   \lesssim \|f\|_{L_s^p (\mathbb R^n)}
\end{equation}
does not hold for all $f\in L_s^p (\mathbb R^n)$ when
\begin{enumerate}
  \item $n=1$ and $s<1/(2p)$ ;
  \item $n\geq 2$, $1\leq p<2$ and $s<n(1/p-1/2)$;
  \item $n\geq 2$, $2\leq p< (2n-1)/(n-1)$ and $s<(1/2-n)(1/2-1/p)+1/4$.
\end{enumerate}
\end{theorem}

For dimensions $n\ge 2$, Vega \cite{V} constructed a counterexample showing that
the pointwise convergence does not hold for all $f\in H^s (\mathbb R^n)$, where $s<1/4$.
Since the Plancherel identity no longer holds when $p \neq 2$,
we must directly estimate the $L^p$ norm of $\mathcal{F}^{-1} ((1+|\cdot|^2)^{s/2})\hat{f})$
in our counterexample construction, which constitutes the main difficulty.
In one spatial dimension, we extend the counterexample given by Dahlberg and Kenig \cite{DK}
from $H^s$ spaces to Bessel potential spaces $L_s^p$.

\textbf{Notation.}
Throughout this paper, for nonnegative quantities $A$ and $B$, we write $A\lesssim B$ to denote that $A\leq CB$ for some positive constant $C$ independent of the relevant variables.
We write $A\gtrsim B$ if $B\lesssim A$, and $A\sim B$ if both $A\lesssim B$ and $A\gtrsim B$ hold. Finally, $A\ll B$ or $B\gg A$ means that $A\leq cB$ for a sufficiently small absolute constant $c>0$.
For a function $f$, we denote its Fourier transform by $\widehat{f}$ or $\mathcal{F}(f)$, and its inverse Fourier transform by $f^{\vee}$ or $\mathcal{F}^{-1}(f)$, which are defined respectively by
\begin{equation*}
\mathcal{F}(f)(\xi) = \int_{\mathbb{R}^n} e^{-ix \cdot \xi} f(x) \mathrm d x \quad \text{and} \quad \mathcal{F}^{-1}(f)(x) = \frac{1}{(2\pi)^n} \int_{\mathbb{R}^n} e^{ix \cdot \xi} f(\xi)
\mathrm d\xi.
\end{equation*}

\section{Preliminaries}\label{s2}
In this section, we record several fundamental propositions that serve as the technical basis for the analysis of oscillatory integrals and operator bounds.

\subsection{Estimates for oscillatory integrals}
To establish our results, we first recall several crucial estimates for oscillatory integrals.
In the one dimensional case, we consider the asymptotic behavior of the integral
\[
I(\lambda) = \int_a^b e^{i\lambda\phi(x)}\psi(x)\, \mathrm{d} x,
\]
for large positive $\lambda$, where $\phi$ is a real-valued smooth phase function,
and $\psi$ is a complex-valued smooth function with compact support in $(a,b)$.
The following result is known as the stationary phase method.

\begin{proposition}{\cite[Proposition 3, Page 334]{St}}
Suppose $k \geq 2$, and
\[
\phi(x_0) = \phi'(x_0) = \dots = \phi^{(k-1)}(x_0) = 0 ,
\]
while $\phi^{(k)}(x_0) \neq 0$.
If the support of $\psi$ contains only one critical point $x_0$ of $\phi$
(i.e., $\phi'(x_0) = 0$ and $\phi'(y) \neq 0$ for all $y\in \mathrm{supp}\,\psi\setminus \{x_0\}$),
then for all nonnegative integers $N$ and $r$, we have
\[
\left( \frac{\mathrm{d}}{\mathrm{d} \lambda} \right)^r \left[ I(\lambda) - \frac{1}{\lambda^{1/k}} \sum_{j=0}^{N} \frac{a_j}{\lambda^{j/k}} \right]
= O\left( \frac{1}{\lambda^{r+(N+1)/k}} \right) \quad \text{as } \lambda \to \infty.
\]
\end{proposition}

According to \cite[Page 337]{St}, we have $a_0=e^{i\pi/4}(2\pi /\phi'' (x_0))^{1/2} \psi(x_0)$.
Consequently, setting $N=r=0$ and $k=2$, we arrive at the following conclusion.

\begin{corollary}\label{cor:s1t1}
Suppose $\phi(x_0) = \phi'(x_0) =0$ and $\phi''(x_0)\neq 0$.
If $\phi'(y) \neq 0$ for all $y\in \mathrm{supp}\,\psi\setminus \{x_0\}$, then
\[
I(\lambda)=e^{i\pi/4}\biggr(\frac{2\pi}{\phi'' (x_0)}\biggr )^{1/2}\frac{\psi(x_0)}{\lambda^{1/2}}
 +O\left(\frac{1}{\lambda^{1/2}}\right) \quad \text{as } \lambda \to \infty.
\]
\end{corollary}

The stationary phase method applies primarily to integrals with a critical point,
whereas for the non-critical case, we typically use the following Van der Corput's lemma.

\begin{proposition}{\cite[Corollary, Page 334]{St}}\label{pro:s1t1}
Let $-\infty < a < b < \infty$, and suppose that $|\phi^{(k)}(x)| \geq 1$ holds for all $x \in [a,b]$.
If either $k=1$ and $\phi'$ is monotonic on $(a,b)$, or $k \geq 2$,
then there exists a constant $C_k$ such that
\[
\left| \int_a^b e^{i\lambda \phi(x)} \psi(x) \, \mathrm{d}x \right|
\leq  \frac{C_k}{\lambda^{1/k}} \left( \int_a^b |\psi'(x)| \, \mathrm{d}x + \|\psi\|_{L^{\infty}(\mathbb{R})}\right).
\]
\end{proposition}

In higher dimensions, we utilize the following estimate due to H\"{o}rmander.
It provides an explicit finite expansion for the integral along with strict bounds.

\begin{proposition}\cite[Theorem 7.7.5]{H}\label{H775}
Let $K\subset \mathbb{R}^{n}$ be a compact set, $X$ an open neighborhood of $K$ and $k$ a positive integer. If $\psi\in C_0^{2k}(K)$, $\phi\in C^{3k+1}(X)$ and $\mathrm{Im}\, \phi\geq 0$ in $X$,
$\mathrm{Im}\, \phi(x_0)=0$, $\nabla \phi(x_0)=0$, $\det \nabla^{2}\phi(x_0)\neq 0$, $\nabla \phi\neq 0$ in $K\setminus \{x_0\}$ then
\begin{align*}
&\left|\int e^{i\lambda \phi(x)}\psi(x)\mathrm{d}x-\frac{(2\pi)^{n/2} e^{n\pi i / 4} e^{i\lambda \phi(x_0)}}{\lambda^{n/2} \det(\nabla^{2}\phi(x_0))^{1/2}}\sum_{j<k}\frac{L_j}{\lambda^{j}} \psi\right|\\
&\leq \frac{C}{\lambda^{k}}\sum_{|\alpha|\leq 2k}\sup |D^{\alpha}\psi|, \quad \lambda>0.
\end{align*}
Here $C$ is bounded when $\phi$ stays in a bounded set in $C^{3k+1}(X)$ and $|x-x_0|/|\nabla \phi(x)|$ has a uniform bound. With
\begin{equation*}
g_{x_0}(x)=\phi(x)-\phi(x_0)-\langle \nabla^{2}\phi(x_0)(x-x_0),x-x_0\rangle/2,
\end{equation*}
the operators $L_j$ are defined by
\begin{equation*}
L_j \psi=\sum_{\nu-\mu=j}\sum_{2\nu\geq 3\mu}\frac{1}{i^{j} 2^{\nu}\mu! \nu!}\left\langle \frac{D}{\nabla^{2}\phi(x_0)},D\right\rangle^{\nu}(g_{x_0}^{\mu}\psi)(x_0).
\end{equation*}
\end{proposition}

Furthermore, when evaluating the exact asymptotic behavior around a nondegenerate critical point, we use the classical stationary phase method, which gives the following full asymptotic expansion.

\begin{proposition}\cite[Page 344, Proposition 6]{St}\label{hnde}
Suppose $\phi(x_0)=0$, and $\phi$ has a nondegenerate critical point at $x_0$. If $\psi$ is supported in a sufficiently small neighborhood of $x_0$, then
\begin{equation*}
\int_{\mathbb{R}^{n}}e^{i\lambda \phi(x)}\psi(x)\mathrm{d}x\sim \frac{1}{\lambda^{n/2}}\sum_{j=0}^{\infty} \frac{a_j}{\lambda^j},\quad \text{as}\,\, \lambda\to\infty,
\end{equation*}
where each coefficient $a_j$ depends on finitely many derivatives of $\phi$ and $\psi$ at $x_0$. In particular,
\begin{equation*}
a_0=\psi(x_0)\cdot (2\pi)^{n/2}\prod_{j=1}^{n}\frac{1}{e^{3\pi i/4}\mu_j^{1/2}},
\end{equation*}
where $\mu_j$ are the eigenvalues of the Hessian matrix of $\phi$ at $x_0$.
\end{proposition}

\subsection{The Fourier transform of surface measure on $\mathbb S^{n-1}$}
The construction of our counterexample in high-dimensional spaces
also requires estimates for the Fourier transform of surface measure on $\mathbb S^{n-1}$.
Let $\mathrm{d}\sigma$ denote surface measure on $\mathbb S^{n-1}$ for $n\geq 2$
and let $J_\nu$ denote the Bessel function defined by
\[
J_\nu(t) = \frac{(t/2)^\nu}{\Gamma(\nu + 1/2) \Gamma(1/2)}
   \int_{-1}^{+1} e^{i t s} (1 - s^2)^\nu \frac{\mathrm{d}s}{\sqrt{1 - s^2}},
\]
where $\operatorname{Re} \nu > -1/2$ and $t \geq 0$.
Then by \cite[Appendix B.4]{G249}, $\widehat{\mathrm{d}\sigma}$ can be represented as
\[
\widehat{\mathrm{d}\sigma}:=\int_{\mathbb S^{n-1}}e^{-i\xi \cdot \theta}\mathrm{d} \theta
                      =\frac{(2\pi)^{n/2}}{|\xi|^{(n-2)/2}} J_{(n-2)/2}(|\xi|).
\]
We write
\begin{align}\label{eq:s1t1}
  J_\nu(r) &= \sqrt{\frac{2}{\pi r}} \cos\left( r - \frac{\pi \nu}{2} - \frac{\pi}{4} \right) + R_\nu(r) \notag \\
           &=\frac{1}{\sqrt{2\pi r}}\left ( e^{i\left( r - (\pi \nu)/{2} - \pi/4 \right)}
             +e^{-i\left( r - (\pi \nu)/{2} - \pi/4 \right)} \right ) + R_\nu(r),
\end{align}
where $R_\nu(r)$ is given by
\begin{align*}
R_\nu(r)
&= \frac{ r^\nu}{(2\pi)^{1/2} \Gamma\left( \nu + 1/2 \right)} e^{i\left( r - (\pi \nu)/2 - \pi/4 \right)} \int_0^\infty e^{-rt} t^{\nu + 1/2} \left[ \left( 1 + \frac{it}{2} \right)^{\nu - 1/2} - 1 \right] \frac{\mathrm{d}t}{t} \\
&\quad + \frac{ r^\nu}{(2\pi)^{1/2}\Gamma\left( \nu + 1/2 \right)} e^{-i\left( r - (\pi \nu)/2 - \pi/4 \right)} \int_0^\infty e^{-rt} t^{\nu + 1/2} \left[ \left( 1 - \frac{it}{2} \right)^{\nu - 1/2} - 1 \right] \frac{\mathrm{d}t}{t}
\end{align*}
and satisfies
\begin{equation*}
|R_\nu(r)| \leq \frac{C_\nu}{r^{3/2}}
\end{equation*}
whenever $r \geq 1$.
For the proof of equality (\ref{eq:s1t1}), we refer to \cite[Appendix B.8]{G249}.
It follows from (\ref{eq:s1t1}) that
\begin{equation}\label{eq:s1t2}
  \widehat{\mathrm{d}\sigma}=\frac{1}{|\xi|^{(n-1)/2}}\left (C_1 e^{i|\xi|}+ C_2 e^{-i|\xi|}\right )
                        +\frac{C_3}{|\xi|^{(n-2)/2}} R_{(n-2)/2}(|\xi|).
\end{equation}

\subsection{Complex interpolation}
For the interpolation of operator bounds between different $L^p$ spaces, we apply the complex interpolation method for analytic families of linear operators.
This theorem, originally due to Stein, provides the fundamental framework for deducing intermediate operator estimates from endpoint bounds.

\begin{proposition}\cite[Theorem 1.3.7]{G249}\label{interpolation}
Let $T_z$ be an analytic family of linear operators of admissible growth defined on the space of finitely simple functions of a $\sigma$-finite measure space $(X,\mu)$
and taking values in the set of measurable functions of another $\sigma$-finite measure space $(Y,\nu)$. Let $1\leq p_0,p_1,q_0,q_1\leq \infty$ and suppose that $M_0$ and $M_1$
are positive functions on the real line such that for some $\tau$ with $0\leq \tau<\pi$ we have
\begin{equation*}
\sup_{-\infty<y<+\infty}e^{-\tau|y|}\log M_j(y)<\infty
\end{equation*}
for $j=0,1$. Fix $0<\theta<1$ and define $p,q$ by the equations
\begin{equation}\label{si}
\frac{1}{p}=\frac{1-\theta}{p_0}+\frac{\theta}{p_1}\qquad \text{and} \qquad \frac{1}{q}=\frac{1-\theta}{q_0}+\frac{\theta}{q_1}.
\end{equation}
Suppose that for all finitely simple functions $f$ on $X$ we have
\begin{equation*}
\|T_{iy}f\|_{L^{q_0}}\leq M_0(y)\|f\|_{L^{p_0}},
\end{equation*}
\begin{equation*}
\|T_{1+iy}f\|_{L^{q_1}}\leq M_1(y)\|f\|_{L^{p_1}}.
\end{equation*}
Then for all finitely simple functions $f$ on $X$ we have
\begin{equation*}
\|T_{\theta}f\|_{L^{q}}\leq M(\theta)\|f\|_{L^{p}},
\end{equation*}
where for $0<x<1$
\begin{equation*}
M(x)=\exp\left(\frac{\sin(\pi x)}{2}\int_{\mathbb{R}}\left(\frac{\log M_0(t)}{\cosh(\pi t)-\cos(\pi x)}+\frac{\log M_1(t)}{\cosh(\pi t)+\cos(\pi x)}\right)\mathrm{d}t\right).
\end{equation*}
By density, $T_\theta$ has a unique bounded extension from $L^{p}(X,\mu)$ to $L^{q}(Y,\nu)$ when $p$ and $q$ are as in \eqref{si}.
\end{proposition}

\section{Proof of main results}\label{s3}
\begin{proof}[Proof of Theorem \ref{t1}]
For a given Schwartz function $f$ on $\mathbb{R}^{n}$, we define another Schwartz $f_s$ by setting
\begin{equation*}
f_s=((1+|\cdot|^{2})^{s/2}\widehat{f})^{\vee},
\end{equation*}
so that we have $\|f\|_{L_s^{p}(\mathbb{R}^{n})}=\|f_s\|_{L^{p}(\mathbb{R}^{n})}$ if $f\in L_s^{p}(\mathbb{R}^{n})$. Hence, we can write
\begin{equation*}
e^{it\Delta}f(x)=e^{it\Delta}\mathcal{J}^{-s}f_s(x)=(K_{s,t}\ast f_s)(x),
\end{equation*}
where
\begin{equation*}
K_{s,t}(x)=\frac{1}{(2\pi)^{n}}\int_{\mathbb{R}^{n}}\frac{e^{i(x\xi-t|\xi|^{2})}}{(1+|\xi|^{2})^{s/2}}\mathrm{d}\xi.
\end{equation*}
Define $K_s^{\ast}(x)=\sup_{0<t<1}|K_{s,t}(x)|$.

\textbf{Step 1.} We show that $\|K_s^{\ast}\ast f_s\|_{L^{1}(B(0,1))}\lesssim \|f_s\|_{L^{1}(\mathbb{R}^{n})}$ for any $s\geq n/2$.

Without loss of generality, we assume that $n/2 \leq s < n$. Noting that
\begin{equation*}
\|K_{s}^{\ast}\ast f_s\|_{L^{1}(B(0,1))}= \int_{\mathbb{R}^n}\int_{B(0,1)}\sup_{0<t<1}|K_{s,t}(x-y)|dx|f_s(y)|\mathrm{d}y,
\end{equation*}
it suffices to show that
\begin{equation}\label{k1}
\int_{B(y,1)}\sup_{0<t<1}|K_{s,t}(z)|\mathrm{d}z\lesssim 1.
\end{equation}

By symmetry, we assume $z \neq 0$. Let $\phi(\xi)= z \cdot \xi - t|\xi|^{2}$, which has a unique stationary point at $\xi_0=z/(2t)$.
Making the change of variables $\eta=2t\xi/|z|$, we can rewrite the kernel as
\begin{equation*}
K_{s,t}(z)=\frac{1}{(2\pi)^{n}}\left(\frac{\lambda}{|z|}\right)^n\int_{\mathbb{R}^n}\frac{e^{i\lambda\left(\omega \cdot \eta-{|\eta|^{2}}/{2}\right)}}{\left(1+\left({\lambda|\eta|}/{|z|}\right)^{2}\right)^{s/2}}\mathrm{d}\eta,
\end{equation*}
where $\lambda={|z|^{2}}/(2t)$ and $\omega = {z}/{|z|} \in \mathbb{S}^{n-1}$.

We fix a smooth, radial cutoff function $\psi\in C_0^{\infty}(\mathbb{R}^n)$ supported in $\{\eta\in \mathbb{R}^{n}: |\eta|\leq 4\}$, identically equal to $1$ in $\{\eta\in \mathbb{R}^{n}: |\eta|\leq 2\}$ and satisfying $0\leq \psi\leq 1$.
Similarly, we choose a smooth radial function $\rho\in C_0^{\infty}(\mathbb{R}^n)$ with $0\leq \rho\leq 1$ such that $\rho(\eta)=1$ for $|\eta|\leq 1/3$ and $\rho(\eta)=0$ for $|\eta|> 2/3$.
We define a partition of unity by setting $\beta_0(\eta)=\psi(\eta)\rho(\eta)$, $\beta_{\ast}(\eta)=\psi(\eta)(1-\rho(\eta))$ and $\beta_j(\eta)=\psi(\eta/2^{j})-\psi(\eta/2^{j-1})$ for $j\geq 1$.
By construction, these functions satisfy
\begin{equation*}
\beta_{\ast}(\eta)+\beta_0(\eta)+\sum_{j=1}^{\infty}\beta_j(\eta)=1.
\end{equation*}
Applying the triangle inequality, we have
\begin{align*}
|K_{s,t}(z)|&\leq \frac{1}{(2\pi)^{n}}\left(\frac{\lambda}{|z|}\right)^n\left(\left|\int_{\mathbb{R}^n}\frac{e^{i\lambda\left(\omega \cdot \eta-{|\eta|^{2}}/{2}\right)}}{\left(1+\left({\lambda|\eta|}/{|z|}\right)^{2}\right)^{s/2}}\beta_{\ast}(\eta)\mathrm{d}\eta\right|\right.\\
&\left.+\sum_{j=0}^{\infty}\left|\int_{\mathbb{R}^n}\frac{e^{i\lambda\left(\omega \cdot \eta-{|\eta|^{2}}/{2}\right)}}{\left(1+\left({\lambda|\eta|}/{|z|}\right)^{2}\right)^{s/2}}\beta_{j}(\eta)\mathrm{d}\eta\right|\right)\\
&:=\frac{1}{(2\pi)^{n}}\left(I_{\ast}+I_0+\sum_{j=1}^{\infty}I_j\right).
\end{align*}

We first estimate $I_{\ast}$. Let
\begin{equation*}
A_{\ast}(\eta)=\frac{\beta_{\ast}(\eta)}{\left(1+\left({\lambda|\eta|}/{|z|}\right)^{2}\right)^{{s}/{2}}}.
\end{equation*}
When $\lambda \leq 1$, by the fact that $\mathrm{supp}\, \beta_{\ast}\subset \{\eta\in \mathbb{R}^{n}:|\eta|\sim 1\}$, we can directly estimate that
\begin{equation}\label{I*1}
|I_{\ast}| \leq \left(\frac{\lambda}{|z|}\right)^n  \|A_{\ast}(\eta)\|_{L^{\infty}(\mathbb{R}^{n})} \lesssim \left(\frac{\lambda}{|z|}\right)^n \frac{1}{\lambda^{n/2}}\min\left\{ 1,\left(\frac{|z|}{\lambda}\right)^{s}\right\}.
\end{equation}

When $\lambda > 1$, we define $\Phi(\eta) = \omega \cdot \eta - {|\eta|^{2}}/{2}$, which has a unique nondegenerate stationary point at $\eta_0 = \omega$ within $\mathrm{supp}\, \beta_{\ast}$. And its Hessian is exactly $\nabla^2 \Phi(\eta_0) = -I_{n\times n}$.
According to Proposition \ref{H775} and triangle inequality, we can choose an integer $k \geq n/2$, such that
\begin{equation}\label{I*2}
\begin{aligned}
|I_{\ast}| &\lesssim \left(\frac{\lambda}{|z|}\right)^n \left( \left| \frac{(2\pi)^{n/2} e^{n\pi i / 4} }{\lambda^{n/2} \det(\nabla^{2}\phi(\eta_0))^{1/2}} \sum_{j<k} \frac{L_j A_{\ast}(\eta_0)}{\lambda^{j}}  \right| + \frac{1}{\lambda^{k}} \sum_{|\alpha| \leq 2k} \|D^\alpha A_{\ast}\|_{L^\infty(\mathbb{R}^{n})} \right)\\
&\lesssim \left(\frac{\lambda}{|z|}\right)^{n} \left( \left| \frac{1}{\lambda^{{n}/{2}}} \sum_{j<k} \frac{L_j A_{\ast}(\eta_0)}{\lambda^{j}} \right| + \frac{1}{\lambda^{k}} \sum_{|\alpha| \leq 2k} \|D^\alpha A_{\ast}\|_{L^\infty(\mathbb{R}^{n})} \right)\\
&\lesssim \left(\frac{\lambda}{|z|}\right)^n \frac{1}{\lambda^{{n}/{2}}} \sum_{|\alpha| \leq 2k} \|D^\alpha A_{\ast}\|_{L^\infty(\mathbb{R}^{n})}.
\end{aligned}
\end{equation}
We assert that
\begin{equation}\label{DA}
\|D^{\alpha} A_{\ast}\|_{L^{\infty}(\mathbb{R}^{n})}\lesssim \min\left\{ 1,\left(\frac{|z|}{\lambda}\right)^{s}\right\},
\end{equation}
for $|\alpha|\leq 2k$.
Indeed, assume $e_i$ is a unit vector whose $i$-th component is $1$ and all others $0$, where $1\leq i\leq n$. Using the fact that $\mathrm{supp}\, \beta_{\ast}\subset \{\eta\in \mathbb{R}^{n}:|\eta|\sim 1\}$, a direct calculation yields
\begin{align*}
|D^{e_i}A_{\ast}(\eta)|&\sim \left|\frac{\left({\lambda|\eta|}/{|z|}\right)^{2}{\eta_i\beta_{\ast}(\eta)}/{|\eta|^{2}}}{\left(1+\left({\lambda|\eta|}/{|z|}\right)^{2}\right)^{{s}/{2}+1}}+\frac{D^{e_i}\beta_{\ast}\left(\eta\right)}{\left(1+\left({\lambda|\eta|}/{|z|}\right)^{2}\right)^{{s}/{2}}}\right|\\
&\lesssim \left|\frac{\left(\beta_{\ast}(\eta)+D^{e_i}\beta_{\ast}\left(\eta\right)\right)}{\left(1+\left({\lambda|\eta|}/{|z|}\right)^{2}\right)^{{s}/{2}}}\right|.
\end{align*}
By repeated application of the above, we obtain
\begin{equation*}
\|D^{\alpha}A_{\ast}(\eta)\|_{L^{\infty}(\mathbb{R}^{n})}\lesssim \frac{1}{\left(1+\left({\lambda |\eta|}/{|z|}\right)^{2}\right)^{{s}/{2}}}\lesssim \min\left\{ 1,\left(\frac{|z|}{\lambda}\right)^{s}\right\}.
\end{equation*}
Combining \eqref{I*1}, \eqref{I*2} and \eqref{DA}, we obtain
\begin{equation}\label{I*3}
|I_{\ast}|\lesssim \left(\frac{\lambda}{|z|}\right)^n \frac{1}{\lambda^{{n}/{2}}} \min \left\{ 1, \left(\frac{|z|}{\lambda}\right)^{s} \right\} \lesssim \frac{1}{t^{{n}/{2}}}\min \left\{ 1, \left(\frac{t}{|z|}\right)^{s} \right\}.
\end{equation}

We consider two cases for \eqref{I*3} according to $|z|$.
For $|z|> 1$, the condition $t\in (0,1)$ trivially implies $|z|>t$. This yields
\begin{equation}\label{I*4}
|I_{\ast}|\lesssim \frac{t^{s-{n}/{2}}}{|z|^{s}}\leq \frac{1}{|z|^{s}},
\end{equation}
where we use the assumption $s \geq n/2$.
For $|z|\leq 1$, we further consider two subcases depending on $t$. If $0< t< |z|$, the same bound applies.
Conversely, if $t\geq |z|$, the estimate becomes
\begin{equation}\label{I*5}
|I_{\ast}|\lesssim \frac{1}{t^{{n}/{2}}}\leq \frac{1}{|z|^{{n}/{2}}}\leq \frac{1}{|z|^{s}}.
\end{equation}
Combining \eqref{I*4} and \eqref{I*5}, we deduce that
\begin{equation} \label{I*f}
|I_{\ast}|\lesssim \frac{1}{|z|^{s}}.
\end{equation}

Next, we estimate $I_0$. Making the substitution $\xi={\lambda \eta}/{|z|}$, we can rewrite $I_0$ as
\begin{equation*}
I_0=\int_{\mathbb{R}^n}e^{i\left(z \cdot \xi-{|z|^{2}|\xi|^{2}}/(2\lambda)\right)}\frac{\beta_0\left({|z|}\xi/{\lambda}\right)}{\left(1+|\xi|^{2}\right)^{s/2}}\mathrm{d}\xi :=\int_{\mathbb{R}^n}e^{i\phi(\xi)}A_0(\xi)\mathrm{d}\xi.
\end{equation*}
Due to the support of $\beta_0$, the integration is restricted to $|\eta|\leq 2/3$, which means $\left|{|z|\xi}/{\lambda}\right|\leq 2/3$. Consequently, the gradient of $\phi$ admits the bound:
\begin{equation}\label{zsim}
|z|\geq |\nabla\phi(\xi)|=\left|z-\frac{|z|^{2}}{\lambda}\xi\right|\geq \frac{1}{3}|z|.
\end{equation}
For $|z|>1$, we introduce the differential operator
\begin{equation*}
L=\frac{\nabla\phi \cdot \nabla}{i|\nabla\phi|^{2}},
\end{equation*}
which satisfies that $L e^{i\phi(\xi)}=e^{i\phi(\xi)}$. Applying integration by parts $N$ times, we have
\begin{equation*}
I_0 = \int_{\mathbb{R}^n} e^{i\phi(\xi)} (L^{\ast})^N A_0(\xi) \mathrm{d}\xi,
\end{equation*}
where $L^{\ast}$ is given by
\begin{equation*}
L^{\ast}f = -\nabla \cdot \left(f \frac{\nabla\phi}{i|\nabla\phi|^2}\right).
\end{equation*}

Since all derivatives of $\phi$ of order 3 and higher strictly vanish,
direct computation shows that
\begin{equation*}
\left|\nabla \cdot \left(\frac{\nabla\phi}{|\nabla\phi|^{2}}\right)\right|=\left| \frac{\Delta \phi}{|\nabla \phi|^{2}}-2\frac{\nabla \phi \cdot (\nabla^{2} \phi \nabla \phi)}{|\nabla \phi|^{4}}\right|\lesssim \frac{|\nabla^{2}\phi|}{|\nabla \phi|^{2}}.
\end{equation*}
It yields that
\begin{equation*}
|L^{\ast}A_0| \lesssim  \frac{|\nabla^{2}\phi|}{|\nabla \phi|^{2}}|A_0| + \frac{1}{|\nabla \phi|}|\nabla A_0|.
\end{equation*}
Repeated computation shows that
\begin{equation}\label{L*N}
|(L^{\ast})^N A_0| \lesssim \sum_{k=0}^{N} \frac{|\nabla^{2}\phi|^{N-k}}{|\nabla \phi|^{2N-k}}|\nabla^{k}A_0|.
\end{equation}
By the Leibniz rule, $\nabla^k A_0(\xi)$ consists of linear combinations of products of $\nabla^{k-l}(1/(1+|\xi|^2)^{{s}/{2}})$ and $\nabla^l \left(\beta_0\left({|z|\xi}/{\lambda}\right)\right)$.
We assert that for any multi-index $\alpha$ with $|\alpha|=m$, $D^{\alpha}(1/(1+|\xi|^{2})^{s/2})$ is a linear combination of $\xi^{\beta}/(1+|\xi|^{2})^{{s}/{2}+p}$.
The multi-index $\beta$ and the non-negative integer $p$ satisfy $2p-|\beta|=m$ and $p\geq m/{2}$. For $m=1$,  a simple calculation implies
\begin{equation*}
D^{e_i}\left(\frac{1}{(1+|\xi|^{2})^{s/2}}\right)=-\frac{s\xi_i}{(1+|\xi|^{2})^{s/2+1}},
\end{equation*}
and $2-|e_i|=1$. Assuming the assertion holds for some $m\geq 1$, we consider the case $m+1$. By assumption, we just need to check $D^{e_j}(\xi^{\beta}/(1+|\xi|^{2})^{s/2+p})$, which satisfies $2p-|\beta|=m$ and $p\geq m/2$.
If $\beta_j\geq 1$, we have
\begin{equation*}
D^{e_j}\left(\frac{\xi^{\beta}}{(1+|\xi|^{2})^{s/2+p}}\right)=\frac{\beta_j \xi^{\beta-e_j}}{(1+|\xi|^{2})^{s/2+p}} - 2\left(\frac{s}{2}+p\right)\frac{\xi^{\beta+e_j}}{(1+|\xi|^{2})^{s/2+p+1}}.
\end{equation*}
For the first term, we have $2p-|\beta-e_j|=m+1$ and $p=(m+|\beta|)/2\geq (m+1)/2$. For the second term, we have $2(p+1)-|\beta+e_j|=m+1$ and $p+1\geq (m+1)/2$.
If $\beta_j= 0$, $D^{e_j}(x^{\beta}/(1+|x|^{2})^{s/2+p})$ only contains the second term, which also satisfies assumptions. This completes the induction and proves the assertion.
Noting that
\begin{align*}
\left|\frac{\xi^{\beta}}{(1+|\xi|^{2})^{s/2+p}}\right|&= \frac{1}{(1+|\xi|^{2})^{s/2}}\frac{|\xi^{\beta}|}{(1+|\xi|^{2})^{p}}\\
&\leq \frac{1}{(1+|\xi|^{2})^{s/2}}\frac{|\xi|^{\beta}}{(1+|\xi|^{2})^{(m+|\beta|)/2}}\\
&\lesssim \frac{1}{(1+|\xi|^{2})^{(s+m)/{2}}}.
\end{align*}
Hence, we obtain
\begin{equation}\label{de1}
|\nabla^{k-l}(1/(1+|\xi|^{2})^{s/2})| \lesssim (1+|\xi|^{2})^{(l-s-k)/{2}}.
\end{equation}
Note that the derivative $\nabla^l \beta_0\left({|z|\xi}/{\lambda}\right)$ produces a factor $\left({|z|}/{\lambda}\right)^l$ via chain rule. When $l \geq 1$, it is supported entirely in $\left\{ \xi\in \mathbb{R}^{n}:{1}/{3} \leq \left|{|z|\xi}/{\lambda}\right| \leq {2}/{3}\right\}$, which implies
\begin{equation}\label{xisim}
|\xi|\sim \frac{\lambda}{|z|}\sim \frac{|z|}{t}.
\end{equation}
Thus, we have
\begin{equation}\label{de2}
|\nabla^l \beta_0\left({|z|\xi}/{\lambda}\right)| \lesssim \frac{1}{|\xi|^{l}}.
\end{equation}
Then \eqref{de1} and \eqref{de2} yield
\begin{equation}\label{de3}
|\nabla^k A_0(\xi)| \lesssim \sum_{l=0}^{k}\frac{1}{(1+|\xi|^{2})^{(s+k-l)/{2}}}\frac{1}{|\xi|^{l}}.
\end{equation}
Utilizing \eqref{zsim} and \eqref{de3}, we obtain
\begin{equation}\label{de4}
\frac{\left|\nabla^{2}\phi\right|^{N-k}}{|\nabla \phi|^{2N-k}}\left|\nabla^{k}A_0(\xi)\right| \lesssim \frac{1}{|z|^{N}} \left(\frac{t}{|z|}\right)^{N-k} \sum_{l=0}^{k}\frac{1}{(1+|\xi|^{2})^{(s+k-l)/{2}}}\frac{1}{|\xi|^{l}}.
\end{equation}
Recalling that $z>1$ and \eqref{xisim}, it yields $|\xi| \sim |z|/t > 1$.
It is easy to verify that
\begin{equation}\label{tzsim}
\frac{t}{|z|} \lesssim \frac{1}{(1+|\xi|^{2})^{{1}/{2}}}.
\end{equation}
Indeed, the support condition of $\beta_0$ dictates ${t}/{|z|} \leq {1}/(3|\xi|)$. If $|\xi| \geq 1$, this directly implies the bound. If $|\xi| < 1$, then $(1+|\xi|^{2})^{1/2}<\sqrt{2}$. This implies that
\begin{equation*}
\frac{t}{|z|} \leq \frac{1}{|z|}<1 \lesssim \frac{1}{(1+|\xi|^{2})^{{1}/{2}}}.
\end{equation*}
By \eqref{L*N}, \eqref{de4} and \eqref{tzsim}, we can choose $N>s$ such that
\begin{equation}\label{I01}
\begin{aligned}
|I_0| &\lesssim \frac{1}{|z|^{N}}\int_{\mathbb{R}^{n}} \sum_{k=0}^{N} \sum_{l=0}^{k} \left(\frac{t}{|z|}\right)^{N-k}\frac{1}{(1+|\xi|^{2})^{(s+k-l)/{2}}}\frac{1}{|\xi|^{l}}\mathrm{d}\xi\\
&\lesssim \frac{1}{|z|^{N}} \int_{\mathbb{R}^n} \frac{1}{(1+|\xi|^2)^{(s+N)/{2}}} \mathrm{d}\xi\\
&\lesssim \frac{1}{|z|^{N}} \leq \frac{1}{|z|^{s}}.
\end{aligned}
\end{equation}

For $0<|z|\leq 1$, we split the integral $I_0$ into two parts:
\begin{equation}\label{I02}
|I_0| \leq \left|\int_{|\xi|\leq 1/|z|}e^{i\phi(\xi)} A_0(\xi)\mathrm{d}\xi\right| + \left| \int_{|\xi|>1/|z|} e^{i\phi(\xi)} A_0(\xi) \mathrm{d}\xi \right| = I_{01} + I_{02}.
\end{equation}
Clearly, the first part is bounded by
\begin{equation}\label{I03}
I_{01} \leq \int_{|\xi|\leq 1/|z|}\frac{1}{(1+|\xi|^{2})^{{s}/{2}}}\mathrm{d}\xi \lesssim \int_{0}^{1/|z|} r^{n-s-1} \mathrm{d}r \lesssim |z|^{s-n}.
\end{equation}
Applying the same integration by parts $N$ times argument in the case of $|z|>1$ to the second part $I_{02}$ implies that
\begin{equation}\label{I04}
|I_{02}| \leq \sum_{m=1}^{N} \int_{|\xi|=1/|z|} \left| (L^{\ast})^{m-1} A_0(\xi) \frac{\nabla\phi(\xi)}{|\nabla\phi(\xi)|^2} \right| \mathrm{d}\sigma(\xi) + \int_{|\xi|>1/|z|} \left| (L^{\ast})^N A_0(\xi) \right| \mathrm{d}\xi.
\end{equation}
Based on \eqref{L*N}, \eqref{de4} and $t/|z|\lesssim 1/|\xi|$, we have
\begin{align*}
\left|(L^{\ast})^{m-1} A_0(\xi)\right| &\lesssim \frac{1}{|z|^{m-1}} \sum_{k=0}^{m-1}\sum_{l=0}^{k} \frac{1}{(1+|\xi|^{2})^{(s+m-1-l)/{2}}}\frac{1}{|\xi|^{l}}\\
&\lesssim \frac{1}{|z|^{m-1}} \frac{1}{|\xi|^{s+m-1}}.
\end{align*}
By \eqref{zsim}, we deduce that
\begin{equation}\label{I05}
\begin{aligned}
\sum_{m=1}^{N}\int_{|\xi|=1/|z|}\left|(L^{\ast})^{m-1}A_0(\xi)\frac{\nabla \phi(\xi)}{|\nabla \phi(\xi)|^{2}}\right|\mathrm{d}\sigma(\xi)&\lesssim \sum_{m=1}^{N} \frac{1}{|z|^{n-1}} \frac{1}{|z|^{m-1}}\frac{1}{|z|^{1-(m+s)}}\frac{1}{|z|}\\
&\sim |z|^{s-n}.
\end{aligned}
\end{equation}
Choosing $N>s$, a direct calculation yields
\begin{equation}\label{I06}
\begin{aligned}
\int_{|\xi|>1/|z|} |(L^{\ast})^N A_0(\xi)| \mathrm{d}\xi &\lesssim \frac{1}{|z|^{N}} \int_{|\xi|>1/|z|} \frac{1}{|\xi|^{s+N}} \mathrm{d}\xi\\
&\sim \frac{1}{|z|^{N}} \int_{1/|z|}^\infty r^{n-s-N-1} \mathrm{d}r\\
&\sim |z|^{s-n}.
\end{aligned}
\end{equation}
Collecting \eqref{I01}, \eqref{I02}, \eqref{I03}, \eqref{I04}, \eqref{I05} and \eqref{I06}, it gives
\begin{equation}\label{I0f}
|I_0|\lesssim \frac{\chi_{\{|z|\leq 1\}}}{|z|^{n-s}}+\frac{\chi_{\{|z|>1\}}}{|z|^{s}}.
\end{equation}

Finally, we estimate $I_j$ for $j\geq 1$. Recalling that $\mathrm{supp}\, \beta_j\subset \{\eta \in \mathbb{R}^{n}: |\eta|\sim 2^{j}\geq 2\}$ and $\Phi(\eta)=\omega \cdot \eta-{|\eta|^{2}}/{2}$, we have $|\nabla \Phi(\eta)|=|\omega-\eta|\sim 2^{j}$.
Let the differential operator
\begin{equation*}
\tilde{L}=\frac{\nabla\Phi \cdot \nabla}{i\lambda |\nabla\Phi|^{2}},
\end{equation*}
which satisfies $\tilde{L} e^{i\lambda \Phi(\eta)}=e^{i\lambda \Phi(\eta)}$.
We define
\begin{equation*}
\tilde{L}^{\ast}f = -\nabla \cdot \left(f \frac{\nabla\Phi}{i\lambda |\nabla\Phi|^2}\right),
\end{equation*}
and
\begin{equation*}
A_j(\eta)=\frac{\beta_j(\eta)}{\left(1+\left({\lambda|\eta|}/{|z|}\right)^{2}\right)^{{s}/{2}}}.
\end{equation*}

In the case of $|z|> 1$, applying integration by parts $N$ times similarly to the estimate of $I_0$, \eqref{L*N}, \eqref{de1} and chain rule, we obtain
\begin{align*}
|I_j| &\lesssim \left(\frac{\lambda}{|z|}\right)^n \left|\int_{\mathbb{R}^n} e^{i\lambda \Phi(\eta)}(\tilde{L}^{\ast})^N A_j(\eta) d\eta\right|\\
&\lesssim \left(\frac{\lambda}{|z|}\right)^n \frac{1}{\lambda^{N}}\sum_{k=0}^{N}\int_{\mathbb{R}^{n}}\frac{1}{|\nabla \Phi(\eta)|^{2N-k}} \left|\nabla^{k}A_j(\eta)\right|\mathrm{d}\eta\\
&\lesssim \left(\frac{\lambda}{|z|}\right)^n \frac{1}{\lambda^{N}}\sum_{k=0}^{N}\int_{\mathbb{R}^{n}}\frac{1}{|\nabla \Phi(\eta)|^{2N-k}} \sum_{l=0}^{k}\frac{\left|\nabla^{l}\beta_j(\eta)\right|}{\left(1+\left({\lambda|\eta|}/{|z|}\right)^{2}\right)^{(s+k-l)/{2}}}\left(\frac{\lambda}{|z|}\right)^{k-l}  \mathrm{d}\eta\\
&\lesssim \left(\frac{\lambda}{|z|}\right)^n \frac{1}{\lambda^{N}}\sum_{k=0}^{N}\int_{\mathbb{R}^{n}}\frac{1}{|\nabla \Phi(\eta)|^{2N-k}} \frac{1}{\left(1+\left({\lambda|\eta|}/{|z|}\right)^{2}\right)^{{s}/{2}}}\sum_{l=0}^{k} \frac{1}{2^{j(k-l)}} \left|\nabla^{l}\beta_j(\eta)\right| \mathrm{d}\eta\\
&\lesssim \left(\frac{\lambda}{|z|}\right)^n \frac{1}{(2^{j} \lambda)^N}\min\left\{1,\left(\frac{|z|}{2^{j} \lambda}\right)^{s}\right\} 2^{jn}\sum_{k=0}^{N}\frac{1}{2^{j(N-k)}}\\
&\lesssim \frac{t^{N+s-n}}{2^{j(N+s-n)} |z|^{2N+s-n}}.
\end{align*}
Choosing an integer $N \geq n/2$, we have
\begin{equation}\label{Ij1}
|I_j| \lesssim \frac{1}{2^{j(N+s-n)} |z|^{s}}.
\end{equation}

For the remaining case $|z| \leq 1$, the same argument as in \eqref{I*1}, \eqref{I*2} and \eqref{DA} applies. Taking $k \geq n/2$, we obtain
\begin{align*}
|I_j| &\lesssim \left(\frac{\lambda}{|z|}\right)^n \frac{1}{\lambda^{n/2}} \sum_{|\alpha| \le 2k} \left\|D^\alpha A_j\right\|_{L^{\infty}(\mathbb{R}^n)} \\
&\lesssim \left(\frac{\lambda}{|z|}\right)^n \frac{1}{\lambda^{{n}/{2}}} \min\left\{1,\left(\frac{|z|}{2^{j} \lambda}\right)^{s}\right\} \\
&\sim \frac{1}{t^{{n}/{2}}} \min\left\{1,\left(\frac{|z|}{2^{j} \lambda}\right)^{s}\right\}.
\end{align*}
We further consider two subcases based on $2^{j}\lambda/|z|$. If $2^{j}\lambda/|z|>1$, we have
\begin{equation}\label{Ij2}
|I_j| \lesssim \frac{1}{t^{{n}/{2}}} \left( \frac{t}{2^j|z|} \right)^{s} = \frac{t^{s-\frac{n}{2}}} { 2^{js}|z|^{s}} \lesssim \frac{1}{ 2^{jn/2}|z|^{s}}.
\end{equation}
Conversely, if $2^{j}\lambda/|z|<1$, we obtain
\begin{equation}\label{Ij3}
|I_j| \lesssim \frac{1}{t^{{n}/{2}}} \lesssim \frac{1}{|z|^{s}2^{j{n}/{2}}}.
\end{equation}
By \eqref{Ij1}, \eqref{Ij2} and \eqref{Ij3}, we deduce that for $j\geq 1$,
\begin{equation*}
|I_j| \lesssim \frac{\chi_{\{|z|\leq 1\}}}{|z|^{s}2^{j{n}/{2}}} + \frac{\chi_{\{|z|>1\}}}{|z|^{s}2^{j(N+s-n)}}.
\end{equation*}
Summing over $j \geq 1$ directly yields
\begin{equation}\label{Ijf}
\sum_{j=1}^{\infty}|I_j|\lesssim \frac{1}{|z|^{s}}.
\end{equation}

Combining \eqref{I*f}, \eqref{I0f} and \eqref{Ijf} and recalling $n/2 \leq s < n$, we conclude that
\begin{align*}
\int_{B(y,1)} \sup_{0<t<1} |K_{s,t}(z)|\mathrm{d}z &\lesssim \int_{B(y,1)}\left(\frac{1}{|z|^{s}}+\frac{1}{|z|^{n-s}}\right)\chi_{\{|z|\leq 1\}}\mathrm{d}z+\int_{B(y,1)}\frac{\chi_{\{|z|>1\}}}{|z|^{s}}\mathrm{d}z\\
&\lesssim \int_{0}^{1} r^{n-1-s}+r^{s-1} \mathrm{d}r + \int_{B(y,1)} \mathrm{d}z\\
&\lesssim 1.
\end{align*}
This establishes the desired bound of \eqref{k1}.

We must point out that $K_s^{\ast}$ is integrable on $\mathbb{R}^{n}$ for all $s>n$. Indeed, for each $|x|\leq 1$ and $t\in (0,1)$, we have
\begin{equation*}
|K_{s,t}(x)|\lesssim \int_{\mathbb{R}^{n}} \frac{1}{(1+|\xi|^{2})^{{s}/{2}}}\mathrm{d}\xi \lesssim 1.
\end{equation*}
Similar to the discussion above, we obtain $|K_s^{\ast}(x)|\lesssim 1/|x|^{s}$ for $|x|>1$. Thus, the conclusion holds.

\textbf{Step 2.} We linearize the maximal operator $S^{\ast}f_s=\sup_{0<t<1}\left|e^{it\Delta}\mathcal{J}^{-s}f_s\right|$.

Let $t(\cdot)$ be an arbitrary Lebesgue measurable function from $B(0,1)$ to $(0,1)$, we define the linearized operator family $T_z$ acting on Schwartz functions $f_s\in \mathcal{S}(\mathbb{R}^{n})$ by
\begin{equation*}
T_z f_s(x)=e^{it(x)\Delta}\mathcal{J}^{-z}f_s(x)=\left. e^{it\Delta}\mathcal{J}^{-z}f_s(x)\right|_{t=t(x)}.
\end{equation*}
We consider $z$ within the complex strips defined by the critical regularity indices. Clearly, the map $z\mapsto \int_{B(0,1)}(T_zf_s(x))g(x)dx$ is analytic on the interior of these strips and continuous on the boundary for any Schwartz functions $f_s$, $g$.

\textbf{Step 3.} Denote $c_n=n/(2n+2)$. We prove that for $\mathrm{Re}\, z=c_n+\varepsilon_1$, $T_{c_n+\varepsilon_1+iy}$ maps $L^{2}(\mathbb{R}^{n})$ to $L^{2}(B(0,1))$ with constant $M_0(y)=C_0$, which is independent of $y$ and $t(x)$.

Invoking the result of Kenig, Ponce and Vega \cite[Theorem 2.5]{KPV}, Du, Guth and Li \cite[Theorem 1.1]{DGL} and Du and Zhang \cite[Theorem 1.2]{DZ}, we have
\begin{equation*}
\left\| \sup_{0<t<1}\left|e^{it\Delta}f\right|\right\|_{L^{2}(B(0,1))}\lesssim \|\mathcal{J}^{c_n+\varepsilon_1}f\|_{L^{2}(\mathbb{R}^{n})}.
\end{equation*}
Substituting $f$ by $\mathcal{J}^{-c_n-\varepsilon_1-iy}f_s$, we get
\begin{equation*}
\|T_{c_n+\varepsilon_1+iy}f_s\|_{L^{2}(B(0,1))}\lesssim \left\| \sup_{0<t<1}\left|e^{it\Delta}\mathcal{J}^{-c_n-\varepsilon_1-iy}f_s\right|\right\|_{L^{2}(B(0,1))}\lesssim \|\mathcal{J}^{-iy}f_s\|_{L^{2}(\mathbb{R}^{n})}.
\end{equation*}
Note that $\mathcal{J}^{-iy}$ coincides with a Fourier multiplier operator with $(1+|\cdot|^{2})^{-iy/2}$, whose modulus is $1$. Applying Plancherel's theorem, we have
\begin{equation*}
\|T_{c_n+\varepsilon_1+iy}f_s\|_{L^{2}(B(0,1))}\lesssim \|f_s\|_{L^{2}(\mathbb{R}^{n})}.
\end{equation*}
Thus, we can set $M_0(y)=C_0$ for all $y\in \mathbb{R}^{n}$.

\textbf{Step 4.} We establish the endpoint bound for $p=1$. Specifically, we show that for $\mathrm{Re}\, z= n/2$, $T_{n/2+iy}$ maps $L^{1}(\mathbb{R}^{n})$ to $L^{1}(B(0,1))$.

We consider the $L^{1}(\mathbb{R}^{n})$ endpoint at $\mathrm{Re}\, z=n/2$. Following the careful analysis in Step 1, replacing $s$ with $n/2+iy$, the derivatives of $1/(1+|\cdot|^{2})^{(n/2+iy)/2}$ will produce polynomials in $y$.
Specifically, differentiating $1/(1+|\cdot|^{2})^{(n/2+iy)/2}$ up to $N$ times introduces a growth factor of order $(1+|y|)^{N}$. Tracking this through our arguments yields
\begin{equation*}
\|T_{n/2+iy}f_s\|_{L^{1}(B(0,1))}\lesssim (1+|y|)^{N}\|f_s\|_{L^{1}(\mathbb{R}^{n})}.
\end{equation*}
Thus, the operator norm is bounded by $M_1(y)=C_1(1+|y|)^{N}$.

\textbf{Step 5.} We complete the proof by applying Proposition \ref{interpolation}.

The bounds $M_0(y)=C_0$ and $M_1(y)=C_1(1+|y|)^{N}$ clearly satisfy the admissible growth condition $\sup_y e^{-\tau|y|}\log M_j(y)<\infty$ for any $0<\tau<\pi$ and $j=0,1$. We interpolate between the $L^{1}(\mathbb{R}^{n})\to L^{1}(B(0,1))$ boundary at $\mathrm{Re}\, z=n/2$
and the $L^{2}(\mathbb{R}^{n})\to L^{2}(B(0,1))$ boundary at $\mathrm{Re}\, z=c_n+\varepsilon_1$. Fixed target space $L^{p}$ for $p\in [1,2]$, we determine the interpolation parameter $\theta\in [0,1]$ satisfying
\begin{equation*}
\frac{1}{p}=\frac{1-\theta}{2}+\frac{\theta}{1}=\frac{1+\theta}{2}.
\end{equation*}
This yields that $\theta=2/p-1$. By proposition \ref{interpolation}, the operator $T_{z_\theta}$ is bounded from $L^{p}(\mathbb{R}^{n})$ to $L^{p}(B(0,1))$ at the point $z_\theta=(1-\theta)(c_n+\varepsilon_1+iy_1)+\theta(n/2+iy_0)$. Let $\varepsilon_1\to 0$.
The required regularity index $s(p)$ corresponds to the real part of $z_\theta$:
\begin{equation*}
s(p)=\mathrm{Re}\, z_\theta=\left(1-\theta\right)\frac{n}{2(n+1)}+\theta\frac{n}{2}=n\left(\frac{1}{p}-\frac{1}{2}\right)+\frac{n}{n+1}\left(1-\frac{1}{p}\right).
\end{equation*}

Ultimately, we have shown that for any $s>s(p)$, $p\in[1,2]$ and any measurable function $t(x)$, the linearized operator satisfies
\begin{equation*}
\left\|e^{it(x)\Delta}\mathcal{J}^{-s}f_s\right\|_{L^{p}(B(0,1))}\lesssim \|f_s\|_{L^{p}(\mathbb{R}^{n})}.
\end{equation*}
Thus, we can take the supremum over all $t\in(0,1)$ inside the $L^{p}$ norm. By selecting $t(x)$ to approximate the pointwise supremum, we recover the bound for the nonlinear maximal operator:
\begin{equation*}
\left\|\sup_{0<t<1}\left|e^{it\Delta}\mathcal{J}^{-s}f_s\right|\right\|_{L^{p}(B(0,1))}\lesssim \|f_s\|_{L^{p}(\mathbb{R}^{n})},
\end{equation*}
which is equivalent to \eqref{conv1}. This completes the proof by density of Schwartz functions in $L_s^{p}(\mathbb{R}^{n})$.
\end{proof}

\begin{remark}
Regarding the one-dimensional case, the $L^2$ estimate of Kenig, Ponce, and Vega \cite[Theorem 2.5]{KPV} holds exactly at the endpoint $s = c_1 = 1/4$.
Therefore, we can set $\varepsilon_1 = 0$ in Step 3. Consequently, the complex interpolation in Step 5 directly yields the endpoint estimate $s \geq 1/(2p)$ for all $1 \leq p \leq 2$, which coincides with the statement in Theorem \ref{t1}.
\end{remark}

\begin{proof}[Proof of Theorem \ref{infty_s}]
We rewrite the solution to the Schr\"odinger equation at $t_0$ and $x=0$ as a convolution:
\begin{equation}\label{conv_rep}
e^{it_0\Delta}f(0) = (K_{s,t_0} \ast f_s)(0) = \int_{\mathbb{R}^n} K_{s,t_0}(-y) f_s(y) \mathrm{d}y,
\end{equation}
where the kernel is given by
\begin{equation*}
K_{s,t_0}(-y) = \frac{1}{(2\pi)^{n}} \int_{\mathbb{R}^n} \frac{e^{-i(y \cdot \xi + t_0|\xi|^2)}}{(1+|\xi|^2)^{{s}/{2}}} \mathrm{d}\xi.
\end{equation*}
Making the translation $u = {2t_0\xi}/{|y|} + {y}/{|y|}$, the integral becomes
\begin{equation}\label{Kt0}
K_{s,t_0}(-y) = \frac{1}{(2\pi)^{n}} e^{i\lambda} \left(\frac{\lambda}{t_0}\right)^{(n-s)/2} \int_{\mathbb{R}^n} \frac{e^{-i\lambda |u|^2}}{ \left({t_0}/{\lambda} + |u - \omega|^2 \right)^{{s}/{2}}} \mathrm{d}u,
\end{equation}
where $\lambda={|y|^{2}}/(4t_0)$ and $\omega={y}/{|y|}$.
We introduce a smooth, radial cutoff function $\psi \in C_0^\infty(\mathbb{R}^n)$ satisfying $\psi(u) = 1$ for $|u| \leq \varepsilon$ and $\psi(u) = 0$ for $|u| > 2\varepsilon$, where $\varepsilon\ll 1$. Denoting
\begin{equation*}
A(u)=\frac{1}{\left( {t_0}/{\lambda} + |u - \omega|^2 \right)^{{s}/{2}}},
\end{equation*}
the integral in \eqref{Kt0} can be split into two parts:
\begin{equation}\label{IID}
\begin{aligned}
\int_{\mathbb{R}^n} e^{-i\lambda |u|^2} A(u) \mathrm{d}u &= \int_{\mathbb{R}^n} e^{-i\lambda |u|^2} A(u) \psi(u) \mathrm{d}u + \int_{\mathbb{R}^n} e^{-i\lambda |u|^2} A(u) (1-\psi(u)) \mathrm{d}u\\
&:= I_1 + I_2.
\end{aligned}
\end{equation}

For $I_1$, we denote $\phi(u) = -|u|^2$, which satisfies $\phi(0)=0$. The origin $u=0$ is a nondegenerate critical point, and its Hessian matrix $\nabla^2\phi(0) = -2I_{n\times n}$ has eigenvalues $\mu_1 = \dots = \mu_n = -2$.
Applying Proposition \ref{hnde}, we have
\begin{equation*}
I_1 = \frac{a_0}{\lambda^{{n}/{2}}}  + O\left(\frac{1}{\lambda^{{n}/{2}+1}}\right), \qquad \text{as}\,\, \lambda \to \infty.
\end{equation*}
By Taylor expansion at $\lambda\to \infty$, $a_0$ is given by
\begin{equation*}
a_0 = A(0)(2\pi)^{{n}/{2}} \prod_{j=1}^n \frac{1}{e^{3\pi i/4}\mu_j^{1/2}}= \frac{\pi^{{n}/{2}} }{e^{i{\pi n}/{4}}\left( {t_0}/{\lambda}+1 \right)^{s/2}} =\frac{\pi^{{n}/{2}}} {e^{i{\pi n}/{4}}}+O\left(\frac{1}{\lambda}\right).
\end{equation*}
Hence, we can write
\begin{equation}\label{II1}
I_1=\frac{\pi^{{n}/{2}} }{e^{i{\pi n}/{4}}\lambda^{n/2}}+O\left(\frac{1}{\lambda^{n/2+1}}\right).
\end{equation}

For $I_2$, we define the differential operator
\begin{equation*}
L = \frac{i u}{2\lambda |u|^2} \cdot \nabla,
\end{equation*}
which satisfies $L e^{-i\lambda|u|^{2}}= e^{-i\lambda|u|^{2}}$. Let
\begin{equation*}
L^{\ast}f=-\nabla \cdot \left(\frac{i u}{2\lambda |u|^2} f\right).
\end{equation*}
Choosing an integer $N > \max\{s, n-s\}$, and applying integration by parts $N$ times similar to the discussion of $I_0$ in proof of Theorem \ref{t1}, \eqref{L*N} and \eqref{de1}, we have
\begin{equation}\label{II2D}
\begin{aligned}
|I_2| &=\left|\int_{\mathbb{R}^{n}} e^{-i\lambda |u|^{2}}\left(L^{\ast}\right)^{N}\left(A(u)(1-\psi(u))\right)\mathrm{d}u\right|\\
&\lesssim \frac{1}{\lambda^{N}}\int_{\mathbb{R}^{n}} \sum_{k=0}^{N} \frac{1}{|u|^{2N-k}}\sum_{l=0}^{k} \frac{\left|\nabla^{l}(1-\psi(u))\right|}{\left({t_0}/{\lambda}+|u-\omega|^{2}\right)^{(s+k-l)/2}}\mathrm{d}u\\
&=\frac{1}{\lambda^{N}}\int_{\mathbb{R}^{n}} \sum_{k=1}^{N} \frac{1}{|u|^{2N-k}}\sum_{l=1}^{k} \frac{\left|\nabla^{l}(1-\psi(u))\right|}{\left({t_0}/{\lambda}+|u-\omega|^{2}\right)^{(s+k-l)/2}}\mathrm{d}u\\
&+\frac{1}{\lambda^{N}}\sum_{k=0}^{N} \int_{\mathbb{R}^{n}} \frac{1}{|u|^{2N-k}}\frac{|1-\psi(u)|}{\left({t_0}/{\lambda}+|u-\omega|^{2}\right)^{(s+k)/2}}\mathrm{d}u\\
&:=J_1+J_2.
\end{aligned}
\end{equation}
Note that $\nabla^{l}(1-\psi)$ does not vanish only on $\{u\in \mathbb{R}^{n}: \varepsilon\leq |u|\leq 2\varepsilon\}$ for every $l\geq 1$. In this case, there exist a constant $c_0>0$ such that
\begin{equation*}
|u-\omega|\geq |\omega|-|u|\geq 1-2\varepsilon>c_0.
\end{equation*}
Thus, we have
\begin{equation}\label{II21}
\begin{aligned}
J_1&=\frac{1}{\lambda^{N}}\sum_{k=1}^{N}\sum_{l=1}^{k} \int_{\varepsilon\leq |u|\leq 2\varepsilon} \frac{1}{|u|^{2N-k}} \frac{\left|\nabla^{l}(1-\psi(u))\right|}{\left({t_0}/{\lambda}+|u-\omega|^{2}\right)^{(s+k-l)/2}} \mathrm{d}u\\
&\lesssim \frac{1}{\lambda^{N}} \sum_{k=1}^{N}\sum_{l=1}^{k} \int_{\varepsilon\leq |u|\leq 2\varepsilon} \frac{1}{|u|^{2N-k}} |u-\omega|^{l-s-k} \mathrm{d}u\\
&\leq \frac{1}{\lambda^{N}} \sum_{k=1}^{N} \int_{\varepsilon\leq |u|\leq 2\varepsilon} \frac{1}{|u|^{2N-k}} \sum_{l=1}^{k} c_0^{l-s-k} \mathrm{d}u\\
&\sim \frac{1}{\lambda^{N}}.
\end{aligned}
\end{equation}
For $J_2$, we separate into two parts:
\begin{equation}\label{II22}
\begin{aligned}
J_2&=\frac{1}{\lambda^{N}}\sum_{k=0}^{N} \int_{|u-\omega|>1/2} \frac{1}{|u|^{2N-k}}\frac{|1-\psi(u)|}{\left({t_0}/{\lambda}+|u-\omega|^{2}\right)^{(s+k)/2}}\mathrm{d}u\\
&+\frac{1}{\lambda^{N}}\sum_{k=0}^{N} \int_{|u-\omega|\leq 1/2} \frac{1}{|u|^{2N-k}}\frac{|1-\psi(u)|}{\left({t_0}/{\lambda}+|u-\omega|^{2}\right)^{(s+k)/2}}\mathrm{d}u\\
&:=J_{21}+J_{22}.
\end{aligned}
\end{equation}
Recalling that $\mathrm{supp}\, (1-\psi)=\{u\in \mathbb{R}^{n}: |u|> \varepsilon\}$, we have
\begin{equation}\label{II23}
J_{21}\lesssim \frac{1}{\lambda^{N}} \sum_{k=0}^{N} \int_{\substack{|u-\omega|>1/2\\ |u|>\varepsilon}} \frac{2^{s+k}}{|u|^{2N-k}} \mathrm{d}u \lesssim \frac{1}{\lambda^{N}}.
\end{equation}
In the case of $|u-\omega|\leq 1/2$, we have $1/2\geq |u-\omega|\geq |\omega|-|u|$, which yields $|u|\geq 1/2$. Select $\lambda\gg 1$ such that $t_0/\lambda+|u-\omega|^{2}<1$. Since $N>n-s$, we deduce that
\begin{equation}\label{II24}
\begin{aligned}
J_{22}&\lesssim \frac{1}{\lambda^{N}} \sum_{k=0}^{N}\int_{|u-\omega|\leq 1/2} \frac{2^{2N-k}}{\left({t_0}/{\lambda}+|u-\omega|^{2}\right)^{(s+k)/2}}\mathrm{d}u \\
&\lesssim \frac{1}{\lambda^{N}} \int_{|u-\omega|\leq 1/2} \frac{1}{\left({t_0}/{\lambda}+|u-\omega|^{2}\right)^{(s+N)/2}} \mathrm{d}u\\
&\sim \frac{1}{\lambda^{N}} \int_{0}^{1/2}\frac{r^{n-1}}{\left({t_0}/{\lambda}+r^{2}\right)^{(s+N)/2}} \mathrm{d}r\\
&=\frac{1}{\lambda^{N}} \left(\frac{t_0}{\lambda}\right)^{(n-s-N)/2} \int_{0}^{\sqrt{\lambda}/(2\sqrt{t_0})} \frac{\rho^{n-1}}{(1+\rho^{2})^{(s+N)/2}}\mathrm{d}\rho\\
&\sim \frac{1}{\lambda^{(n-s+N)/2}}.
\end{aligned}
\end{equation}
Combining \eqref{II2D}, \eqref{II21}, \eqref{II22}, \eqref{II23} and \eqref{II24}, we conclude that
\begin{equation}\label{II2f}
|I_2| \lesssim \frac{1}{\lambda^{N}} + \frac{1}{\lambda^{(n-s+N)/2}}\lesssim \frac{1}{\lambda^{(n-s+N)/2}}, \qquad \text{as}\,\, \lambda \to \infty.
\end{equation}

Recalling $\lambda = {|y|^2}/(4t_0)$, we deduce from \eqref{Kt0}, \eqref{IID}, and \eqref{II1} that
\begin{equation*}
K_{s,t_0}(-y) =c_1 \frac{e^{i|y|^{2}/(4t_0)}}{|y|^{s}}+c_2e^{i|y|^{2}/(4t_0)}\left(O\left(\frac{1}{|y|^{s+2}}\right)+|y|^{n-s}I_2\right),
\end{equation*}
where $c_1 = 2^{s-n}t_0^{s-n/2}/(\pi^{n/2}e^{i\pi n/4})$, and $c_2=2^{s-2n}t_0^{s-n}/\pi^{n}$.
Now, we define
\begin{equation*}
f_s(y) =
\begin{cases}
e^{-i{|y|^2}/(4t_0)} \overline{\operatorname{sgn}(c_1)}, & \text{if } |y| > R,\\
0, & \text{if } |y| \leq R,
\end{cases}
\end{equation*}
where $R \gg 1$ . Clearly, $\|f_s\|_{L^\infty(\mathbb{R}^{n})} = 1$.

Substituting $f_s$ into \eqref{conv_rep}, and by \eqref{II2f},  we deduce that
\begin{align*}
\left| \int_{\mathbb{R}^n} K_{s,t_0}(-y) f_s(y) \mathrm{d}y \right|
&= \left| \int_{|y| \geq R} \frac{|c_1|}{|y|^{s}}+c_2\overline{\operatorname{sgn}(c_1)}\left(O\left(\frac{1}{|y|^{s+2}}\right)+|y|^{n-s}I_2\right) \mathrm{d}y \right| \\
&\gtrsim \int_{|y| \geq R} \frac{1}{|y|^{s}} \mathrm{d}y - \int_{|y| \geq R} O\left(\frac{1}{|y|^{s+2}}\right) \mathrm{d}y-\int_{|y| \geq R} |y|^{n-s}|I_2| \mathrm{d}y\\
&\gtrsim \int_{|y| \geq R} \frac{1}{|y|^{s}} \mathrm{d}y - \int_{|y| \geq R} \frac{1}{|y|^{s+2}} \mathrm{d}y-\int_{|y| \geq R} \frac{1}{|y|^{N}} \mathrm{d}y\\
&= \int_{|y| \geq R} \frac{1}{|y|^{s}}\left(1-\frac{1}{|y|^{2}}-\frac{1}{|y|^{N-s}}\right) \mathrm{d}y.
\end{align*}
Since $N>s$, we can choose $R$ large enough such that $1-1/|y|^{2}-1/|y|^{N-s} \geq 1/2$ for all $|y| \geq R$, we obtain
\begin{equation*}
\left| \int_{\mathbb{R}^n} K_{s,t_0}(-y) f_s(y) \mathrm{d}y \right| \gtrsim \int_{|y| \geq R} \frac{1}{|y|^{s}} \mathrm{d}y \sim \int_{R}^{\infty} r^{n-1-s} \mathrm{d}r.
\end{equation*}
Since $s \leq n$, the above integral diverges, which completes the proof.
\end{proof}

\begin{proof}[Proof of Theorem \ref{p_s}]
Assume \eqref{p>2} hold for $p\geq 2$, $s<n(1-2/p)$ and $f\in L_s^{p}(\mathbb{R}^{n})$. Letting $\lambda\gg 1$ and selecting a smooth, radial function $\varphi \in C_c^{\infty}(\mathbb{R}^{n})$ which is supported on $\{\xi\in \mathbb{R}^{n}: 1/2\leq |\xi|\leq 2\}$, we define
\begin{equation*}
\widehat{f_\lambda}(\xi)=e^{i|\xi|^{2}/2}\varphi\left(\frac{\xi}{\lambda}\right).
\end{equation*}
Noting that $\mathrm{supp}\, \varphi(\cdot/\lambda)\subset \{\xi\in \mathbb{R}^{n}: |\xi|\sim \lambda\}$, we obtain
\begin{equation}\label{ee1}
\|f_\lambda\|_{L^{p}_s(\mathbb{R}^{n})}\sim \lambda^{s}\|f_{\lambda}\|_{L^{p}(\mathbb{R}^{n})}.
\end{equation}
Denoting $g_\lambda(x)=(\varphi(\cdot/\lambda))^{\vee}(x)=\lambda^{n}{\varphi}^{\vee}(\lambda x)$, we write
\begin{equation*}
f_\lambda(x)=\frac{1}{(2\pi)^{n}}\int_{\mathbb{R}^{n}}e^{ix\cdot \xi}e^{i|\xi|^{2}/2}\varphi\left(\frac{\xi}{\lambda}\right)d\xi=e^{-i\Delta/2}g_\lambda(x).
\end{equation*}
For $|x|\leq 3\lambda$, by the dispersive estimate, it follows
\begin{equation}\label{ee2}
|f_\lambda(x)|=|e^{-i\Delta/2}g_\lambda(x)|\lesssim \|g_\lambda\|_{L^{1}(\mathbb{R}^{n})}=\|\varphi^{\vee}\|_{L^{1}(\mathbb{R}^{n})}\lesssim 1.
\end{equation}
For $|x|>3\lambda$, making change of $\xi/\lambda=u$, we have
\begin{equation*}
f_\lambda(x)=\frac{\lambda^{n}}{(2\pi)^{n}}\int_{\mathbb{R}^{n}}e^{i\lambda^{2}(2u/\lambda+|u|^{2}/2)}\varphi(u) \mathrm{d}u:=\frac{\lambda^{n}}{(2\pi)^{n}}\int_{\mathbb{R}^{n}}e^{i\lambda^{2}\psi(u)}\varphi(u) \mathrm{d}u.
\end{equation*}
Let the differential operator $L$ be defined as
\begin{equation*}
L=\frac{\nabla \psi \cdot \nabla}{i\lambda^{2}|\nabla \psi|^{2}},
\end{equation*}
and let
\begin{equation*}
L^{\ast}f=-\nabla \cdot \left(\frac{\nabla \psi}{i\lambda^{2}|\nabla \psi|^{2}}f\right).
\end{equation*}
Note that $L e^{i\lambda \psi}=e^{i\lambda \psi}$. Picking an integer $N>n(p+1)/(2p)$ and applying integration by parts $N$ times similar to the argument of $I_0$ in proof of Theorem \ref{t1} and \eqref{L*N}, we deduce that
\begin{equation}\label{ee3}
\begin{aligned}
|f_\lambda(x)|&=\frac{\lambda^{n}}{(2\pi)^{n}}\left|\int_{\mathbb{R}^{n}}e^{i\lambda^{2}\psi(u)}\left(L^{\ast}\right)^{N}\varphi(u)\mathrm{d}u\right|\\
&\lesssim \lambda^{n-2N}\int_{\mathbb{R}^{n}} \sum_{k=0}^{N}\frac{1}{|\nabla \psi(u)|^{2N-k}}|\nabla^{k} \varphi(u)|\mathrm{d}u\\
&\leq \lambda^{n-2N}\int_{\mathbb{R}^{n}} \sum_{k=0}^{N}\frac{1}{(x/\lambda)^{2N-k}}|\nabla^{k} \varphi(u)|\mathrm{d}u\\
&\lesssim \frac{\lambda^{n-N}}{|x|^{N}}.
\end{aligned}
\end{equation}
Combining \eqref{ee2} and \eqref{ee3}, we have
\begin{align*}
\|f_\lambda\|_{L^{p}(\mathbb{R}^{n})}^{p}&=\int_{|x|\leq 3\lambda} |f_\lambda(x)|^{p} \mathrm{d}x + \int_{|x|>3\lambda} |f_\lambda(x)|^{p} \mathrm{d}x\\
&\lesssim \int_{|x|\leq 3\lambda} \mathrm{d}x+ \int_{|x|>3\lambda} \frac{\lambda^{(n-N)p}}{|x|^{Np}} \mathrm{d}x\\
&\sim \lambda^{n}+\lambda^{np-2Np+n}\int_{|x|>3} \frac{1}{|x|^{Np}}\mathrm{d}x\\
&\lesssim \lambda^{n}.
\end{align*}
Recalling \eqref{ee1}, we obtain
\begin{equation}\label{ee4}
\|f_\lambda\|_{L^{p}_s(\mathbb{R}^{n})}\lesssim \lambda^{s+n/p}.
\end{equation}

A direct calculation yields that
\begin{equation*}
e^{i\Delta/2}f_\lambda(x)=\frac{1}{(2\pi)^{n}}\int_{\mathbb{R}^{n}}e^{ix\cdot \xi}\varphi\left(\frac{\xi}{\lambda}\right)\mathrm{d}\xi=\lambda^{n} \varphi^{\vee}(\lambda x).
\end{equation*}
Since $\varphi^{\vee}(0)>0$, by continuity, there exists a constant $c>0$, such that $|\varphi^{\vee}(\lambda x)|\geq c>0$ for all $|x|\lesssim \lambda^{-1}$.
Denote $\alpha_0=c\lambda^{n}/2$. For any $x$ satisfying $|x|\lesssim \lambda^{-1}$, we have
\begin{equation*}
\lambda^{n}|\varphi^{\vee}(\lambda x)|\geq c\lambda^{n}>\alpha_0.
\end{equation*}
Recalling $\lambda\gg 1$, it clearly follows that $\{x\in \mathbb{R}^{n}:|x|\lesssim 1/\lambda\}\subset B(0,1)$. Hence, we deduce that
\begin{equation*}
\left\{x\in \mathbb{R}^{n}:|x|\lesssim \lambda^{-1}\right\}\subset \left\{x\in B(0,1): |\lambda^{n} \varphi^{\vee}(\lambda x)|>\alpha_0\right\}.
\end{equation*}
This yields that
\begin{equation}\label{ee5}
\begin{aligned}
\left\|\sup_{0<t<1}\left|e^{it\Delta}f_\lambda\right|\right\|_{L^{p,\infty}(B(0,1))}&\geq \|\lambda^{n}\varphi^{\vee}(\lambda \cdot)\|_{L^{p,\infty}(B(0,1))}\\
&=\sup_{\alpha>0} \alpha \left|\left\{x\in B(0,1): |\lambda^{n} \varphi^{\vee}(\lambda x)|>\alpha\right\}\right|^{1/p}\\
&\geq \alpha_0 \left|\left\{x\in B(0,1): |\lambda^{n} \varphi^{\vee}(\lambda x)|>\alpha_0\right\}\right|^{1/p}\\
&\geq \alpha_0 \left|\left\{x\in \mathbb{R}^{n}:|x|\lesssim \lambda^{-1}\right\}\right|^{1/p}\\
&\sim \lambda^{n(1-1/p)}.
\end{aligned}
\end{equation}
By \eqref{ee4} and \eqref{ee5}, we conclude that
\begin{equation*}
\lambda^{n(1-1/p)}\lesssim \left\|\sup_{0<t<1}\left|e^{it\Delta}f_\lambda\right|\right\|_{L^{p,\infty}(B(0,1))}\lesssim \|f_\lambda\|_{L^{p}_s(\mathbb{R}^{n})}\lesssim \lambda^{s+n/p}.
\end{equation*}
Since $\lambda\gg 1$, we deduce that $s\geq n(1-2/p)$, which contradicts the assumption. This completes the proof.
\end{proof}

To prove Theorem \ref{p<2}, we need the following lemma.

\begin{lemma}\label{lem:s1t1}
$\mathrm{(i)}$ Suppose $\kappa>0$ and $f\in \mathcal{S}(\mathbb{R})$ satisfies $\mathrm{supp}\,\widehat{f}\subseteq [-\kappa,\kappa]$. Then
\[
\left| \int_{\mathbb{R}}f(x) \mathrm{d}x \right|\leq \frac{\kappa^2}{2} \int_{\mathbb{R}}|f(x)| x^2 \mathrm{d}x.
\]
$\mathrm{(ii)}$ There exist some $\kappa_0 >2\sqrt{2}$ and $\phi \in \mathcal{S}(\mathbb{R})$ satisfying
$\mathrm{supp}\,\widehat{\phi}\subseteq [-\kappa_0,\kappa_0]$ and
\[
4\int_{\mathbb{R}}|\phi(x)| x^2 \mathrm{d}x < \left| \int_{\mathbb{R}}\phi(x) \mathrm{d}x \right|.
\]
\end{lemma}
\begin{proof}
  $\mathrm{(i)}$ Let $f_e (x)=(f(x)+f(-x))/2$. Then $\widehat{f_e}$ is an even Schwartz function satisfying $\widehat{f_e}'(0)=0$. By Taylor's formula, there exists $c\in (0,\kappa)$ such that
  \begin{align*}
    \widehat{f_e}(\kappa) & =\widehat{f_e}(0)+ \kappa\widehat{f_e}'(0)+\frac{\kappa^2}{2}\widehat{f_e}''(c)\\
                      & =\widehat{f_e}(0)+ \frac{\kappa^2}{2}\widehat{f_e}''(c).
  \end{align*}
  It follows from $\widehat{f_e}(\kappa)=0$ that
  \begin{align*}
    \left| \int_{\mathbb{R}}f(x) \mathrm{d}x \right| & =|\widehat{f_e}(0)|= \frac{\kappa^2}{2}|\widehat{f_e}''(c)|\\
     & \leq  \frac{\kappa^2}{2}\sup_{\xi\in \mathbb{R}}\left| \int_{\mathbb{R}}f_e (x) x^2 e^{-ix\xi} \mathrm{d}x \right|\\
     & \leq  \frac{\kappa^2}{2}\int_{\mathbb{R}} |f_e (x)| x^2 \mathrm{d}x\\
     & \leq \frac{\kappa^2}{2}\int_{\mathbb{R}} |f (x)| x^2 \mathrm{d}x.
  \end{align*}
  $\mathrm{(ii)}$ Fix an arbitrary $\psi \in \mathcal{S}(\mathbb{R})$ such that
  $\mathrm{supp}\,\widehat{\psi}\subseteq [-1,1]$ and $\widehat{\psi}(0)=1$.
  We define a family of functions $\{\phi_{\kappa}\}$ by $\phi_{\kappa}(y)=\kappa \psi(\kappa y)$.
  Then we have $\widehat{\phi_{\kappa}}(\xi)=\widehat{\psi}(\xi/\kappa)$
  and $\mathrm{supp}\,\widehat{\phi_{\kappa}}\subseteq [-\kappa,\kappa]$.
  It is easy to see that
  \[
  \int_{\mathbb{R}}\phi_{\kappa}(y) \mathrm{d}y=\widehat{\phi_{\kappa}}(0)=\widehat{\psi}(0)=1
  \]
  and
  \[
   \int_{\mathbb{R}}|\phi_{\kappa}(y)|y^2 \mathrm{d}y=\frac{1}{\kappa^2}\int_{\mathbb{R}}|\psi(y)|y^2 \mathrm{d}y.
  \]
  Therefore, there exists a sufficiently large $\kappa_0$ such that
  \[
  4\int_{\mathbb{R}}|\phi_{\kappa_0}(y)| y^2 \mathrm{d}y < \left| \int_{\mathbb{R}}\phi_{\kappa_0}(y) \mathrm{d}y \right|.
  \]
  It follows from the result of $\mathrm{(i)}$ that $\kappa_0 >2\sqrt{2}$.
\end{proof}

Now we are in a position to prove Theorem \ref{p<2}.
\begin{proof}[Proof of Theorem \ref{p<2}]
$\mathrm{(i)}$ We first consider the one-dimensional case.
By Lemma \ref{lem:s1t1} $\mathrm{(ii)}$, there exist some $\kappa_0 >2\sqrt{2}$ and $\eta \in \mathcal{S} (\mathbb{R})$
such that $\mathrm{supp}\, \widehat{\eta} \subseteq [-\kappa_0, \kappa_0]$ and
\[
4\int_{\mathbb{R}}|\eta(x)| x^2 \mathrm{d}x < \left| \int_{\mathbb{R}}\eta(x) \mathrm{d}x \right|.
\]
We define the family of functions $\{f_{\nu}\}$ by
\[
f_{\nu}(x)=\eta\left(\frac{x}{\nu}\right)e^{\frac{2ix}{\nu^2}}, \quad 0<\nu< \frac{1}{100\kappa_0}.
\]
Then we have
\begin{align*}
\widehat{f_{\nu}}(\xi)
& = \int_{\mathbb{R}} \eta\left(\frac{x}{\nu}\right)e^{\frac{2ix}{\nu^2}}e^{-ix\xi}\mathrm{d} x
  =\nu \int_{\mathbb{R}} \eta\left(\lambda\right)e^{\frac{2i\lambda}{\nu}}e^{-i\lambda(\nu\xi)}\mathrm{d} \lambda\\
& =\nu \widehat{\eta}\left(\nu\xi-\frac{2}{\nu}\right).
\end{align*}
Therefore, we obtain
\[
\left| \mathcal{F}^{-1}\left((1+\xi^2)^{s/2}\widehat{f_\nu}(\xi)\right)(y)\right|
= \frac{\nu}{2\pi}\left|\int_{\mathbb{R}}e^{iy\xi} (1+\xi^2)^{s/2}
   \widehat{\eta}\left(\nu\xi-\frac{2}{\nu}\right)\mathrm{d} \xi\right|.
\]
By a change of variables $\theta=\nu\xi-2/\nu$, we get
\begin{align*}
   \left| \mathcal{F}^{-1}\left((1+\xi^2)^{s/2}\widehat{f_\nu}(\xi)\right)(y)\right|
& = \frac{1}{2\pi}\left|\int_{\mathbb{R}}e^{i(y/\nu)\theta} \left(1+\frac{(\nu\theta+2)^2}{\nu^4}\right)^{s/2}
     \widehat{\eta}(\theta)\mathrm{d} \theta\right| \\
& = \frac{1}{\nu^{2s}}\left| \mathcal{F}^{-1}((\nu^4+(\nu\theta +2)^2)^{s/2}
     \widehat{\eta}(\theta))\left( \frac{y}{\nu} \right)\right|.
\end{align*}
Thus we have
\begin{align*}
\|f_{\nu}\|_{L_s^p(\mathbb{R})}
&=\left\|\mathcal{F}^{-1} \left((1+\xi^2)^{s/2}\widehat{f_\nu}(\xi)\right)
  \right\|_{L^{p}(\mathbb{R})}\\
&=\nu^{1/p-2s}\left\| \mathcal{F}^{-1} ((\nu^4+(\nu\theta +2)^2)^{s/2} \widehat{\eta}(\theta)) \right\|_{L^{p}(\mathbb{R})}.
\end{align*}
We claim that there exists a constant $C$ independent of $\nu$, such that
\begin{equation}\label{eq:s2t7}
  \| \mathcal{F}^{-1} ((\nu^4+(\nu\theta +2)^2)^{s/2} \widehat{\eta}(\theta)) \|_{L^{p}(\mathbb{R})} \leq C.
\end{equation}
Indeed, it suffices to show that
\[
\left| \mathcal{F}^{-1}((\nu^4+(\nu\theta +2)^2)^{s/2} \widehat{\eta}(\theta))(y)\right|
 \lesssim \min\left\{1,\frac{1}{|y|^2}\right\}.
\]
Since the boundedness is trivial, we only need to consider the case where $|y|>1$.
By integration by parts, we obtain
\begin{align*}
 \mathcal{F}^{-1}( (\nu^4+(\nu\theta +2)^2)^{s/2} \widehat{\eta}(\theta) )(y)
&=\frac{1}{2\pi} \int_{-\kappa_0}^{\kappa_0} e^{iy\theta}(\nu^4+(\nu\theta +2)^2)^{s/2}\widehat{\eta}(\theta)\mathrm{d} \theta\\
&=\frac{-1}{2\pi y^2} \int_{-\kappa_0}^{\kappa_0} e^{iy\theta}  g_{\nu}''(\theta) \mathrm{d} \theta ,
\end{align*}
where $g_{\nu}(\theta)=(\nu^4+(\nu\theta +2)^2)^{s/2}\widehat{\eta}(\theta)$.
Recall that $0<\nu<1/(100\kappa_0)$ and $\theta \in [-\kappa_0,\kappa_0]$.
It is easy to verify that $|g_{\nu}''(\theta)|\leq M$, where $M$ is independent of $\nu$.
Then we arrive at (\ref{eq:s2t7}). Thus we have $\|f_{\nu}\|_{L_s^p(\mathbb{R})}\lesssim \nu^{1/p-2s}$.

For any $x$ belonging to $(1/4,1)$, we take $t_x=\nu^2 x/4$.
Arguing as in \cite[Theorem, Page 207]{DK}, we write
\begin{align*}
|e^{it_x\Delta}f_{\nu}(x)|
&=\frac{1}{2\pi}\left|\int_{\mathbb{R}}e^{ix\xi}e^{-it_x \xi^2}\widehat{f_{\nu}}(\xi) \mathrm{d} \xi \right|\\
&= \left|\frac{1}{e^{\pi i/4}\sqrt{4\pi t_x}} \int_{\mathbb{R}} e^{i(x-y)^2/(4t_x)} f_{\nu}(y) \mathrm{d} y \right|\\
&\sim \frac{1}{\nu\sqrt{x}}\left| \int_{\mathbb{R}}e^{iy^2/(\nu^2 x)} \eta\left(\frac{y}{\nu}\right) \mathrm{d} y\right|\\
&\geq \frac{1}{\sqrt{x}}\left| \left|\int_{\mathbb{R}} \eta(y) \mathrm{d} y \right|
       -\left|\int_{\mathbb{R}} \eta(y) (e^{iy^2/x}-1) \mathrm{d} y \right| \right|\\
&\gtrsim \left|\int_{\mathbb{R}} \eta(y) \mathrm{d} y \right|-\frac{1}{x}\int_{\mathbb{R}} |\eta(y)| y^2 \mathrm{d} y\\
&\gtrsim 1,
\end{align*}
where we have used Lemma \ref{lem:s1t1} $\mathrm{(ii)}$ in the last inequality.
Therefore, we obtain
\[
\sup_{0<t<1} |e^{it\Delta} f_{\nu}(x)| > C_0 ,\quad \forall x\in (1/4,1),
\]
where $C_0$ is independent of $\nu$.
Suppose on the contrary that (\ref{eq:s2t1}) holds for all $f\in L_s^p (\mathbb{R})$ with $s<1/(2p)$. Then
\[
C_0 |(1/4,1)|^{1/q}
\leq \sup_{\delta >0}\delta \left| \left\{x\in B(0,1):\sup_{0<t<1}|e^{it\Delta} f_{\nu}(x)|>\delta \right\} \right|^{1/q}
 \lesssim \|f_{\nu}\|_{L_s^p (\mathbb{R})}.
\]
That is, $\nu^{1/p-2s}\gtrsim 1$. Letting $\nu\rightarrow 0$ gives the desired contradiction.

$\mathrm{(ii)}$ Next, we consider the case $n\geq 2$ and $1\leq p<2$.
Let $\eta \in C_0^\infty (\mathbb{R})$ be a nonnegative function
which is supported on $[1,4]$ and satisfies $\eta (2\sqrt{2})=1$.
We define the sequence $\{\eta_k\}_{k\geq 1}$ by $\eta_k (\xi)=\eta (\xi/2^k)$.
It follows that each $\eta_k$ is supported on $[2^k, 2^{k+2}]$ and
\[
\left| \frac{d^j}{d \xi^j} \eta_k (\xi) \right| \lesssim \frac{1}{2^{kj}}, \quad j\geq 0.
\]
Let $\widehat{f_k}(\xi)=\eta_k (|\xi|)$.
By equality (\ref{eq:s1t2}), we obtain
\begin{align*}
  |e^{it\Delta} f_k (x)| & \sim \left| \int_{\mathbb{R}^n}e^{-it|\xi|^2}e^{ix\cdot \xi} \eta_k (|\xi|)\mathrm{d} \xi \right| \\
    & =  \left| \int_{2^k}^{2^{k+2}} e^{-itr^2}\eta_k (r) r^{n-1}
          \left( \int_{\mathbb S^{n-1}}e^{i(rx)\cdot u}\mathrm{d} \sigma(u) \right) \mathrm{d} r\right| \\
    & = \left| \int_{2^k}^{2^{k+2}} e^{-itr^2}\eta_k (r) r^{n-1}
         \left( \frac{C_1 e^{ir|x|}+ C_2 e^{-ir|x|}}{|rx|^{(n-1)/2}}
         +\frac{C_3 R_{(n-2)/2}(r|x|)}{|rx|^{(n-2)/2}}  \right) \mathrm{d} r\right| \\
    & \geq \frac{|C_1|}{|x|^{(n-1)/2}}
       \left|\int_{2^k}^{2^{k+2}} e^{-i(tr^2-r|x|)} r^{(n-1)/2} \eta_k (r) \mathrm{d} r \right| \\
    & \qquad -\frac{|C_2|}{|x|^{(n-1)/2}}
       \left|\int_{2^k}^{2^{k+2}} e^{-i(tr^2+r|x|)} r^{(n-1)/2} \eta_k (r) \mathrm{d} r \right| \\
    & \qquad -\frac{|C_3|}{|x|^{(n-2)/2}}
       \left|\int_{2^k}^{2^{k+2}} e^{-itr^2} r^{n/2} \eta_k (r) R_{(n-2)/2}(r|x|) \mathrm{d} r \right|\\
    & :=I_1(x)-I_2(x)-I_3(x).
\end{align*}
We first estimate $I_3 (x)$.
Note that $|R_{(n-2)/2}(r|x|)|\lesssim 1/(r|x|)^{3/2}$ (see \cite[Appendix B.8]{G249} for example).
Therefore, for any $x\in B(0,1)$ satisfying $1/2 \leq |x|< 1$, we have
\begin{equation}\label{eq:s2t2}
  I_3 (x)\lesssim \int_{2^k}^{2^{k+2}} r^{(n-3)/2} \mathrm{d} r \lesssim 2^{k(n-1)/2}.
\end{equation}
We proceed to estimate $I_1 (x)$.
For each $x\in B(0,1)$ satisfying $1/2 \leq |x|< 1$, we take $t_x =|x|/2^{k+5/2}$.
Consequently, we get
\[
  I_1 (x)=\frac{|C_1|}{|x|^{(n-1)/2}}
           \left|\int_{2^k}^{2^{k+2}} e^{-i2^k |x|(r^2/2^{2k+5/2}-r/2^k)} r^{(n-1)/2} \eta_k (r) \mathrm{d} r \right|.
\]
Let $\lambda =2^k |x|$ and $\phi_1 (r)=r^2/2^{2k+5/2}-r/2^k$.
It is easy to see that $\phi_1' (2^{k+3/2})=0$ and $\phi_1'' (2^{k+3/2})=1/2^{2k+3/2}$.
Then by Corollary \ref{cor:s1t1}, we obtain
\begin{align*}
  I_1 (x) & =\frac{|C_1|}{|x|^{(n-1)/2}} \left|\int_{2^k}^{2^{k+2}}
          e^{-i\lambda(\phi_1 (r)-\phi_1 (2^{k+3/2}))} r^{(n-1)/2} \eta_k (r) \mathrm{d} r \right| \\
          & \gtrsim \frac{1}{|x|^{(n-1)/2}} \left[
          (2^{2k+3/2})^{1/2}(2^{k+3/2})^{(n-1)/2} \eta_k(2^{k+3/2})(2^k |x|)^{-1/2}-O((2^k |x|)^{-1/2}) \right].
\end{align*}
It follows from $\eta_k(2^{k+3/2})=\eta(2\sqrt{2})=1$ that
\begin{equation}\label{eq:s2t3}
  I_1 (x)\gtrsim \frac{2^{k+k(n-1)/2-k/2}}{|x|^{n/2}}\gtrsim 2^{nk/2}.
\end{equation}
Finally, we estimate $I_2 (x)$. We write
\[
I_2(x) = \frac{|C_2|}{|x|^{(n-1)/2}} \left| \int_{2^k}^{2^{k+2}}
   e^{-i|x|\left( r^2/2^{k+5/2} + r \right)} r^{(n-1)/2} \eta_k(r) \, \mathrm{d} r \right|.
\]
Let $\phi_2(r) = r^2/2^{k+5/2} + r$. Then $\phi_2'(r) > 1$ for all $r$ belonging to $(2^k, 2^{k+2})$.
Applying Proposition \ref{pro:s1t1}, we obtain for $1/2 \leq |x|< 1$,
\begin{align}\label{eq:s2t4}
I_2(x)
&\lesssim \frac{1}{|x|^{(n+1)/2}} \left( \int_{2^k}^{2^{k+2}}
           \left| \frac{n-1}{2} r^{(n-3)/2} \eta_k(r) + r^{(n-1)/2} \eta_k'(r) \right| \mathrm{d} r
           +\left\| r^{(n-1)/2} \eta_k(r) \right\|_{L^{\infty}} \right) \notag \\
&\lesssim 2^{k(n-1)/2}+2^{k(n+1)/2-k}+2^{k(n-1)/2} \lesssim 2^{k(n-1)/2}.
\end{align}
Combining (\ref{eq:s2t2}), (\ref{eq:s2t3}) and (\ref{eq:s2t4}),
for $1/2 \leq |x|< 1$ and sufficiently large $k$, we have
\begin{equation}\label{eq:s2t5}
  \sup_{0<t<1}|e^{it\Delta} f_k (x)|\geq |e^{it_x\Delta}f_k (x)| \gtrsim 2^{nk/2}.
\end{equation}
To complete the proof, we only need to verify that
\begin{equation}\label{eq:s2t6}
 \|f_k\|_{L_s^p} \lesssim 2^{k(s+n(1-1/p))}.
\end{equation}
In fact,
\[
\|f_k\|_{L_s^p(\mathbb{R}^n)}
:= \left\| \mathcal{F}^{-1} \left( \widehat{f}_k(\cdot) (1+|\cdot|^2)^{s/2} \right) \right\|_{L^{p}(\mathbb{R}^n)}
 =  \left\| \mathcal{F}^{-1} \left( \eta_k(|\cdot|) (1+|\cdot|^2)^{s/2} \right) \right\|_{L^p(\mathbb{R}^n)}.
\]
By a change of variables $\nu =\xi/2^k$, we obtain
\begin{align*}
\mathcal{F}^{-1} \left( \eta_k(|\cdot|) (1+|\cdot|^2)^{s/2} \right)(x)
&= \frac{1}{(2\pi)^n} \int_{\mathbb{R}^n} \eta_k(|\xi|) (1+|\xi|^2)^{s/2} e^{i x\cdot \xi} \, \mathrm{d} \xi \\
&= \frac{1}{(2\pi)^n} \int_{\mathbb{R}^n} \eta(|\xi/2^k|) (1+|\xi|^2)^{s/2} e^{i x\cdot \xi} \, \mathrm{d} \xi \\
&= \frac{2^{k(n+s)}}{(2\pi)^n} \int_{\mathbb{R}^n}
    \eta(|\nu|) (1/2^{2k} + |\nu|^2)^{s/2} e^{i (2^k x)\cdot \nu} \, \mathrm{d} \nu.
\end{align*}
Let $g_k (\nu):= \eta(|\nu|) (|\nu|^2 + 1/2^{2k})^{s/2}$. Then we have
\[
\left\| \mathcal{F}^{-1} \left( \eta_k(|\cdot|) (1+|\cdot|^2)^{s/2} \right) \right\|_{L^p(\mathbb{R}^n)}
 \sim 2^{k(n+s-n/p)} \left\| \mathcal{F}^{-1}(g_k) \right\|_{L^p(\mathbb{R}^n)}.
\]
All that remains is to verify that
there exists a constant $C>0$ independent of $k$ such that $\|\mathcal{F}^{-1}(g_k)\|_{L^p(\mathbb{R}^n)} \leq C$.
It suffices to prove $|\mathcal{F}^{-1}(g_k)(x)|\lesssim \min\{1,1/|x|^{n+1}\}$.
Since the boundedness is trivial, we only need to consider the case where $|x|>1$. Note that
\[
e^{i x \cdot \xi} = \frac{1}{(i |x|^2)^{n+1}} \sum_{|\beta| = n+1} C_\beta x^\beta
                     \partial_\xi^\beta \left( e^{i x \cdot \xi} \right).
\]
By integrating by parts, we obtain
\[
\mathcal{F}^{-1}g_k(x)=\frac{1}{(2\pi)^n}\cdot\frac{1}{(i|x|^2)^{n+1}}
 \sum_{|\beta|=n+1}C_\beta' x^\beta\int_{\mathbb{R}^n}\left(\partial_\xi^\beta g_k(\xi)\right) e^{ix\cdot\xi}\mathrm{d} \xi.
\]
Recall that $\eta(|\cdot|)\in C_c^\infty(\mathbb{R}^n)$ with support on $\{\xi\in\mathbb{R}^n:1\leq|\xi|\leq 4\}$.
Therefore, we can verify that
\[
\left|\partial_\xi^\beta g_k(\xi)\right|\leq\frac{C_{\beta,n,s}}{(1+|\xi|)^{n+1}}.
\]
Combining this with $|x^\beta|\lesssim|x|^{|\beta|}=|x|^{n+1}$ gives $\left|\mathcal{F}^{-1}g_k(x)\right|\lesssim1/|x|^{n+1}$.
Then we arrive at the inequality (\ref{eq:s2t6}).

Assume (\ref{eq:s2t1}) holds for all $f\in L_s^p (\mathbb R^n)$ with $s<n(1/p-1/2)$.
Then setting $\delta =2^{nk/2}$ in (\ref{eq:s2t1}) yields
\[
\left| \left\{x\in B(0,1):\sup_{0<t<1}|S_t f_k(x)| \gtrsim 2^{nk/2} \right\} \right|^{1/q}
 \lesssim \frac{1}{2^{nk/2}} \|f_k\|_{L_s^p (\mathbb R^n)}.
\]
It follows from (\ref{eq:s2t5}) and (\ref{eq:s2t6}) that
\begin{align*}\
        |B(0,1)\setminus B(0,1/2)|^{1/q}
    &= \left|\left\{x\in B(0,1)\setminus B(0,1/2):
        \sup_{0<t<1}|e^{it\Delta} f_k(x)| \gtrsim 2^{nk/2} \right\} \right|^{1/q} \\
    &\lesssim \frac{1}{2^{nk/2}} \|f_k\|_{L_s^p (\mathbb R^n)} \lesssim 2^{k(s+n(1/2-1/p))}.
\end{align*}
Letting $k\rightarrow \infty$ gives a contradiction.

$\mathrm{(iii)}$ Finally, we consider the case $n\geq 2$ and $2\leq p< (2n-1)/(n-1)$.
Let $\eta \in C_0^\infty (\mathbb{R})$ be a nonnegative function
which is supported on $[-1,1]$ and satisfies $\eta (0)=1$.
We define the sequence $\{f_k\}_{k\geq 1}$ by $\widehat{f_k} (\xi)=\eta ((|\xi|-2^k)/2^{k/2})$.
Then by the Hausdorff-Young's inequality, we have
\begin{align*}
  \|f_k\|_{L_s^p(\mathbb{R}^{n})} &=\|\mathcal{F}^{-1}((1+|\xi|^2)^{s/2}\widehat{f_k})\|_{L^p(\mathbb{R}^n)}
                     \lesssim  \left\|(1+|\cdot|^2)^{s/2}\eta\left(\frac{|\cdot|-2^k}{2^{k/2}}\right)\right\|_{L^{p'}(\mathbb{R}^n)}\\
                  &=\left( \int_{\{\xi:2^k-2^{k/2}\leq |\xi|\leq 2^k+2^{k/2}\}} (1+|\xi|^2)^{sp'/2}
                     \left|\eta\left(\frac{|\xi|-2^k}{2^{k/2}}\right)\right|^{p'} \mathrm{d} \xi \right)^{1/p'} \\
                  &\lesssim \left( \int_{2^k-2^{k/2}}^{2^k+2^{k/2}} (1+r^2)^{sp'/2} r^{n-1} \mathrm{d} r \right)^{1/p'}\\
                  &\lesssim 2^{ks+k(n-1/2)/p'}.
\end{align*}
It follows from (\ref{eq:s1t2}) that
\begin{align*}
     |e^{it\Delta} & f_k (x)|
   \sim \left|\int_{\mathbb{R}^n}e^{-it|\xi|^2}e^{ix\cdot \xi}
      \eta\left(\frac{|\xi|-2^k}{2^{k/2}}\right)\mathrm{d} \xi \right| \\
    &= \left| \int_{2^k-2^{k/2}}^{2^k+2^{k/2}} e^{-itr^2} \eta\left(\frac{r-2^k}{2^{k/2}}\right) r^{n-1}
        \left( \int_{\mathbb S^{n-1}}e^{i(rx)\cdot u}\mathrm{d} \sigma(u) \right) \mathrm{d} r\right| \\
    &= \left| \int_{2^k-2^{k/2}}^{2^k+2^{k/2}} e^{-itr^2} \eta\left(\frac{r-2^k}{2^{k/2}}\right) r^{n-1}
        \left( \frac{C_1 e^{ir|x|}+ C_2 e^{-ir|x|}}{|rx|^{(n-1)/2}}
        +\frac{C_3R_{(n-2)/2}(r|x|)}{|rx|^{(n-2)/2}}  \right) \mathrm{d} r\right| \\
    &\geq \frac{|C_1|}{|x|^{(n-1)/2}} \left|\int_{2^k-2^{k/2}}^{2^k+2^{k/2}} e^{-i(tr^2-r|x|)} r^{(n-1)/2}
      \eta\left(\frac{r-2^k}{2^{k/2}}\right) \mathrm{d} r \right| \\
    &\qquad -\frac{|C_2|}{|x|^{(n-1)/2}} \left|\int_{2^k-2^{k/2}}^{2^k+2^{k/2}} e^{-i(tr^2+r|x|)} r^{(n-1)/2}
      \eta\left(\frac{r-2^k}{2^{k/2}}\right) \mathrm{d} r \right| \\
    &\qquad -\frac{|C_3|}{|x|^{(n-2)/2}} \left|\int_{2^k-2^{k/2}}^{2^k+2^{k/2}} e^{-itr^2} r^{n/2}
      \eta\left(\frac{r-2^k}{2^{k/2}}\right) R_{(n-2)/2}(r|x|) \mathrm{d} r \right| \\
    &:=II_1(x)-II_2(x)-II_3(x).
\end{align*}
For any $x$ satisfying $1/2\leq |x|<1$, we set $t_x =|x|/2^{k+1}$. Then we have
\[
II_1 (x)=\frac{|C_1|}{|x|^{(n-1)/2}}\left|\int_{2^k-2^{k/2}}^{2^k+2^{k/2}} e^{-i2^k |x|(r^2/2^{2k+1}-r/2^k)} r^{(n-1)/2}
          \eta \left(\frac{r-2^k}{2^{k/2}}\right) \mathrm{d} r \right|.
\]
Let $\lambda =2^k |x|$ and $\varphi (r)=r^2/2^{2k+1}-r/2^k$.
It is easy to see that $\varphi' (2^k)=0$ and $\varphi'' (2^k)=1/2^{2k}$.
Then by Corollary \ref{cor:s1t1}, we obtain
\begin{align*}
  II_1 (x) & =\frac{|C_1|}{|x|^{(n-1)/2}} \left|\int_{2^k-2^{k/2}}^{2^k+2^{k/2}} e^{-i\lambda(\varphi (r)-\varphi(2^k))}
               r^{(n-1)/2} \eta\left(\frac{r-2^k}{2^{k/2}}\right) \mathrm{d} r \right| \\
           & \gtrsim \frac{1}{|x|^{(n-1)/2}} \left[ (2^{2k})^{1/2} 2^{k(n-1)/2}
              \eta(0) (2^k |x|)^{-1/2}-O((2^k |x|)^{-1/2}) \right] \gtrsim 2^{kn/2}.
\end{align*}
Moreover, similarly to the proofs of (\ref{eq:s2t2}) and (\ref{eq:s2t4}),
we have $\max\{II_2 (x), II_3 (x)\}\lesssim 2^{k(n-1)/2}$ for sufficiently large $k$.
Therefore, we obtain
\[
\sup_{0<t<1}|e^{it\Delta} f_k(x)|\geq |e^{it_x\Delta} f_k(x)|\gtrsim 2^{kn/2}
\]
holds for any $1/2\leq |x|<1$.
Suppose on the contrary that (\ref{eq:s2t1}) holds for all $f\in L_s^p (\mathbb{R}^n)$
with $s<(1/2-n)(1/2-1/p)+1/4=(n-1/2)(1/2-1/p')+1/4$. Then
\begin{align*}\
        |B(0,1)\setminus B(0,1/2)|^{1/q}
    &= \left|\left\{x\in B(0,1)\setminus B(0,1/2):
        \sup_{0<t<1}|e^{it\Delta} f_k(x)| \gtrsim 2^{nk/2} \right\} \right|^{1/q} \\
    &\lesssim \frac{1}{2^{nk/2}} \|f_k\|_{L_s^p (\mathbb R^n)} \lesssim 2^{k(s+(n-1/2)/p'-n/2)}.
\end{align*}
Letting $k\rightarrow \infty$ gives the desired contradiction.
\end{proof}

\begin{remark}
It is straightforward to see that for $p=2$,
Theorem \ref{p<2} recovers the necessary conditions for pointwise convergence established by
Dahlberg and Kenig \cite{DK} (for $n=1$) and Vega \cite{V} (for $n\geq 2$), respectively.
\end{remark}


\end{document}